\documentclass[10pt,reqno]{amsart}
\usepackage{amssymb,mathrsfs,color}

\usepackage{wrapfig}

\usepackage{tikz}
\usepgflibrary{fpu}
\usetikzlibrary{arrows,calc,decorations.pathreplacing}
\definecolor{light-gray1}{gray}{0.90}
\definecolor{light-gray2}{gray}{0.80}

\definecolor{deepgreen}{cmyk}{1,0,1,0.5}

\newcommand{\E}{\mathcal{E}}

\newcommand{\LL}{\mathcal{L}}

\newcommand{\HH}{\mathcal{H}}

\newcommand{\NN}{\mathcal{N}}

\def\calK{\mathcal{K}}
\def\f{\frac}

\newcommand{\R}{\mathbb{R}}
\newcommand{\Z}{\mathbb{Z}}

\newcommand{\al}{\alpha}

\newcommand{\ga}{\gamma}
\newcommand{\de}{\delta}
\newcommand{\e}{\varepsilon}
\newcommand{\fy}{\varphi}

\newcommand{\la}{\lambda}

\newcommand{\p}{\partial}

\makeatletter

\newcommand{\Rmnum}[1]{\expandafter\@slowromancap\romannumeral #1@}
\makeatother

\newcommand{\I}{\infty}

\newcommand{\ti}{\widetilde}

\newcommand{\ds}{\displaystyle}

\newcommand{\abs}[1]{\left\lvert{#1}\right\rvert}

\newcommand{\pmat}[1]{\begin{pmatrix} #1 \end{pmatrix}}

\newcommand{\ali}[1]{\begin{align}\begin{split} #1 \end{split}\end{align}}
\newcommand{\ant}[1]{\begin{align*}\begin{split} #1 \end{split}\end{align*}}

\newcommand{\EQ}[1]{\begin{equation}\begin{split} #1 \end{split}\end{equation}}
\newcommand{\Del}[1]{}

\newcommand{\pt}{&}

\def\ti{\tilde}

\numberwithin{equation}{section}

\newtheorem{thm}{Theorem}[section]
\newtheorem{cor}[thm]{Corollary}
\newtheorem{lem}[thm]{Lemma}
\newtheorem{prop}[thm]{Proposition}

\newtheorem{claim}[thm]{Claim}
\theoremstyle{remark}
\newtheorem{rem}{Remark}

\newcommand{\mfor}{{\ \ \text{for} \ \ }}
\newcommand{\mas}{{\ \ \text{as} \ \ }}

\newcommand{\calS}{\mathcal{S}}

 \def\Id{\mathrm{Id}}
 \def\nn{\nonumber}
\def\lan{\langle}
\def\ran{\rangle}
\def\eps{\varepsilon}
\def\glei{\mathrm{eq}}

\begin{document}

\author{Carlos~E.~Kenig}
\author{Andrew~Lawrie}
\author{Wilhelm~Schlag}

\title[Relaxation to harmonic maps]{Relaxation of wave maps exterior to a ball to harmonic maps for all data}

\begin{abstract}
In this paper we establish relaxation of an arbitrary $1$-equivariant wave map from $\R^{1+3}_{t,x}\setminus (\R\times B(0,1))\to S^3$ of finite energy and with
a Dirichlet condition at $r=1$, to the unique stationary harmonic map in its degree class. This settles a recent conjecture of Bizo\'n, Chmaj, Maliborski who
observed this asymptotic behavior numerically. 
\end{abstract}

\thanks{Support of the National Science Foundation DMS-0968472 for the first author, and  DMS-0617854, DMS-1160817 for the third author is gratefully acknowledged. }



\maketitle

\section{Introduction}

In this paper we describe all possible asymptotic dynamics for the $1$-equivariant wave-map equation 
from
$$
\R^{1+3}_{t,x}\setminus (\R\times B(0,1))\to S^3
$$
with a Dirichlet condition on the boundary of the ball~$B(0,1)$, and data of finite energy. 
To be specific, consider the Lagrangian
\[
\LL(U,\p_t U) = \int\limits_{\R^{1+3}\setminus (\R\times B(0,1))} \frac12 \big(-|\p_t U|_g^2 + \sum_{j=1}^3 |\p_j U|_g^2\big)\, dt dx
\]
where $g$ is the round metric on~$S^3$, and we only consider functions for which the boundary of the cylinder~$ \R\times B(0,1)$ gets mapped to a fixed point on~$S^3$, say the north pole. 
 Under the usual $1$-equivariance assumption the Euler-Lagrange equation associated with this Lagrangian becomes
\EQ{\label{WMell}
\psi_{tt} - \psi_{rr} - \frac{2}{r}\psi_r + \frac{\sin(2\psi)}{r^2}=0
} 
where $\psi(t,r)$ measures the angle from the north-pole on~$S^3$. 
The imposed Dirichlet boundary condition is then $\psi(t, 1)=0$ for all $t\in \R$. In other words, we
are considering the Cauchy problem
\EQ{\label{eq:psi wm}
&\psi_{tt} - \psi_{rr} - \frac{2}{r}\psi_r + \frac{\sin(2\psi)}{r^2}=0, \quad r\ge1,\\ 
&\psi(t, 1)=0,\quad \forall\:t,\\
&\psi(0,r)=\psi_0(r),\quad \psi_t(0,r)=\psi_1(r)
}
The conserved energy is
\EQ{\label{ener}
\E(\psi,\psi_t)=\int_1^\infty \frac12\big( \psi_t^2+\psi_r^2+2\frac{\sin^2(\psi)}{r^2}\big) r^2\, dr 
}
Any $\psi(t,r)$  of finite energy  and continuous dependence on $t\in I:= (t_{0},t_{1})$ must satisfy $\psi(t, \infty)=n\pi$ for all $t\in I$ where $n\in \Z$ is fixed. We can restrict to the case $n \ge 0$ since this covers the entire range $n\in\Z$ by the symmetry $\psi\mapsto -\psi$. 
We call $n$ the {\em degree}, and denote by $\E_n$
 the connected component of the metric space of all $\vec \psi=(\psi_0,\psi_1)$ with $\E(\vec\psi)<\I$ and fixed degree~$n$ (of course obeying the boundary condition at $r=1$), i.e., 
 \EQ{
 \E_n :=\{ (\psi_0, \psi_1) \mid  \E( \psi_0, \psi_1)< \infty, \, \,  \psi_0(1) = 0, \, \, \lim_{r \to \I} \psi_0(r) = n \pi\}
 }
 The advantage of this model lies with the fact that removing the unit ball eliminates the scaling symmetry
and also renders the equation subcritical relative to the energy. This subcriticality  immediately implies global wellposedness 
in the energy class. Both of these features are in stark contrast to the same
equation on $1+3$-dimensional Minkowski space, which is known to be super-critical and to develop singularities in finite time, see Shatah~\cite{Shatah} and also Shatah, Struwe~\cite{SSbook}.

Another striking feature of this model, which fails for the $1+2$-dimensional analogue,  lies with the fact that it admits  infinitely many stationary solutions~$(Q_{n}(r), 0)$ which
satisfy $Q_{n}(1)=0$ and $\lim_{r\to\infty}Q_{n}(r)=n\pi$, for each $n\ge1$. These solutions have minimal energy in the degree class~$\E_n$, 
and they are the unique stationary solutions in that class.

The natural space to place the solution into for $n=0$ is
the {\em energy space} $\HH_0:= (\dot H^1_{0}\times L^2)(\R^3_*)$ with norm 
\EQ{\label{eq:HH norm}
\| \vec \psi\|_{\HH_0}^2 := \int_1^\infty (\psi_r^2(r) +  \psi_t^2(r))\, r^2\, dr,\qquad \vec \psi=(\psi, \psi_t)
}
Here, $\R^3_*:= \R^3 \setminus B(0,1)$ and $\dot H^1_{0}(\R^3_*)$ is the completion under the first norm on the right-hand
side of~\eqref{eq:HH norm} of the smooth radial functions on $\{x \in \R^3 \mid \abs{x} >1\}$ with compact support.   For $n\ge1$, we denote $\HH_n:=\E_n-(Q_n,0)$ with ``norm"
\[
\|\vec \psi\|_{\HH_n}:=\| \vec \psi -(Q_n,0)\|_{\HH_0}
\]
The point of this notation is that the boundary condition at $r=\I$ is $ \vec \psi -(Q_n,0)(r)\to 0$ as $r\to\I$. 

The exterior equation \eqref{eq:psi wm} was proposed by Bizon, Chmaj, and Maliborski~\cite{Biz}  as a model in which to study the problem of relaxation to the ground
states given by the various equivariant harmonic maps. In the physics literature, this model was introduced in~\cite{Physics} as an easier alternative to the Skyrmion equation.
Moreover, \cite{Physics} stresses the  
analogy with the damped pendulum which plays an important role in our analysis.  Both~\cite{Biz, Physics} obtain the existence and uniqueness of the ground state
harmonic maps via the phase-plane of the damped pendulum, and they also  observed stability of the linearized equation around the harmonic maps. 
Numerical simulations described in~\cite{Biz} indicated that in each equivariance
class and topological class given by the boundary value~$n\pi$ at  $r=\infty$ {\em every solution} scatters to the unique harmonic map $Q_n$
that lies in this class. In this paper we verify  this conjecture in the $1$-equivariant setting, for all degrees and all data.

Our main result is as follows. It should be viewed as a verification of the {\em soliton resolution conjecture} for this particular case.

\begin{thm}\label{main} 
For any smooth energy data in $\E_n$ there exists a unique global and smooth solution to~\eqref{eq:psi wm} which scatters to 
the harmonic map~$(Q_n,0)$. 
\end{thm}

Scattering here means that  on compact regions in space one has $(\psi,\psi_t)(t) - (Q_{n},0)\to (0,0)$ in the energy topology, or  alternatively
\EQ{\label{psi n scat}
(\psi,\psi_t)(t)=(Q_{n},0)+(\fy,\fy_t)(t)+o_{\HH_n}(1)\quad t\to\infty
}
where $(\fy,\fy_t)\in \HH_0$ solves the linearized version of~\eqref{eq:psi wm}, i.e.,
\EQ{\label{fy eq}
\fy_{tt} - \fy_{rr} - \frac{2}{r}\fy_r + \frac{2}{r^2}\fy=0, \; r\ge1, \:  \fy(t,1)=0
}
We would like to emphasize that only the scattering part of Theorem~\ref{main} is difficult. 

In \cite{LS} the second and third authors established this theorem for degree zero, and also proved asymptotic stability of the $Q_n$ for $n\ge1$.
Here we are able to treat data of all sizes in the higher degree case. As in~\cite{LS} we employ the method of concentration compactness from~\cite{KM06, KM08}.
The main difference from~\cite{LS} lies with the rigidity argument. While the virial identity was the key to rigidity in~\cite{LS} for degree zero (which seems
to be impossible for $n\ge1$), here we follow an alternate route which was developed in a very different context in~\cite{DKM4, DKM5}
for the three-dimensional energy critical nonlinear focusing wave equation. To be specific we rely on the {\em exterior asymptotic energy} arguments developed there. 
A novel feature of our work is that we elucidate the role of the Newton potential as an obstruction to linear  energy estimates exterior to a cone in odd dimensions; in particular we do this 
for $\dim=5$,
which is what is needed for equivariant wave maps in~$\R^3$. It is precisely
this feature which allows us to adapt the rigidity blueprint from~\cite{DKM4, DKM5} to the model under consideration. 

Finally, let us mention that we expect the methods of this paper to carry over to higher equivariance classes as well. 

\section{Preliminaries}\label{prelim}

In this section we discuss the harmonic maps $Q_n$, as well as the reduction of the equivariant wave maps equation to
a semi-linear equation in~$\R^5_*:=\R^5\setminus B(0,1)$ with a Dirichlet condition at $r=1$. 

\subsection{Exterior Harmonic Maps}  In each energy class, $\E_n$ there is a unique finite energy exterior harmonic map, $(Q, 0)=(Q_n, 0)$. In fact $(Q_n, 0)$ can be seen to have minimal  energy in $\E_n$.  An exterior harmonic map is a stationary solution of~\eqref{eq:psi wm}, i.e., 
\begin{align}
\label{hm}
 &Q_{rr} + \frac{2}{r} Q_r = \frac{\sin(2 Q)}{r^2}\\ 
& Q(1)=0, \quad \lim_{r \to \infty} Q(r)= n \pi  \label{bc}
\end{align}

\begin{lem}\label{ode lem}
For all $\alpha\in\R$ there exists a unique solution $Q_\alpha\in \dot H^1(\R^3_*)$ to \eqref{hm} with 
\[
Q_\al(r) =  n \pi -  \al r^{-2} + O(r^{-6})
\]
The $O(\cdot)$ is determined by~$\al$, and vanishes for~$\al=0$. 
Moreover, there exists a unique $\al$ such that $Q_\al(1)=0$, which we denote by $\al_0$. One has $\al_0>0$. 
\end{lem}

The proof of Lemma~\ref{ode lem} is standard so we just sketch an outline below. In order to study solutions to~\eqref{hm} it is convenient to introduce new variables $s= \log(r)$ and $\phi(s) = Q(r)$. With this change of variables we obtain an autonomous differential equation for $\phi$, viz.,
\EQ{\label{eq: phi}
\ddot \phi + \dot\phi= \sin(2 \phi)
 }
which is the equation for a damped pendulum. We can thus reduce matters to  the phase portrait associated to \eqref{eq: phi}. 
Setting $x(s) = \phi(s)$, $y(s) =  \dot\phi(s)$ we rewrite~\eqref{eq: phi} as the system
\EQ{\label{system}
 \pmat{ \dot{x} \\  \dot y}  = \pmat{ y \\ - y+ \sin(2x)} =: X(x, y)
}
and we denote by $\Phi_s$ the flow associated to $X$. The equilibria of \eqref{system}  occur at points $v_{k/2} = (\frac{ k \pi }{2} , 0)$ where $k \in \Z$. For each $\frac{k}{2} = n \in  \Z$  the flow has a saddle with eigenvalues $\la_+ = 1$, $\la_-= -2$, and the corresponding unstable and stable invariant subspaces for the linearized flow are given by the spans of $(1, \la_+)=(1, 1)$, respectively $(1, \la_-) = (1, -2)$. In a neighborhood $V \ni v_n = (n \pi, 0)$ one can define the  $1$-dimensional invariant unstable manifold 
\ant{
W^u_n = \{ (x, y) \in V \mid \Phi_s(x, y) \to v_n\, \,  \textrm{as} \, \, s \to - \infty\}
}
and the $1$-dimensional invariant stable manifold 
\ant{
W^{st}_n = \{ (x, y) \in V \mid \Phi_s(x, y) \to v_n\, \,  \textrm{as} \, \, s \to  \infty\}
}
which are tangent at $v_n$ to the invariant subspaces of the linearized flow. In particular, for each $n$  one can parameterize the stable manifold $W^{st}_n$ by 
\ant{
\phi_{n, \al}(s) = n \pi - \al e^{-2s} + O(e^{-6s})
}
with the parameter $\al$ determining all the coefficients of higher order. This proves the existence of the $Q_{\al}$ in Lemma~\ref{ode lem}. One can show that if the parameter $\al$ satisfies $\al>0$ then  $\phi_{n, \al}(s)$ lies on the branch of the stable manifold which stays below $n\pi$ for all $s \in \R$, i.e., $\phi_{n, \al}(s) < n \pi$ for all $s \in \R$. If $\al=0$ then $\phi_{n, \al}(s) = n \pi$ for all $s$. Finally, if $\al<0$ then $\phi_{n, \al}(s) > n \pi$ for all $s \in \R$.  Different choices of $\al$ correspond to translations  in $s$ along the respective branches of the stable manifold, which is what we mean by uniqueness in the statement of Lemma~\ref{ode lem}.  

To prove the existence and uniqueness of $\al_0$, we note that an analysis of the phase portrait shows that any trajectory with $\al>0$ must have crossed the $y$-axis at some finite time $s_0$, and once it has crossed can never do so again. Note that if the parameter $\al$ satisfies  $\al<0$ then the trajectory can never cross the $y$-axis since in this case $\phi_{n, \al}(s) > n \pi$ for all $s \in \R$. 

Now, fix any $\al_+ >0$ and $\al_-<0$. Passing back to the original variables we have three trajectories
\EQ{ \label{stable m}
&Q_{n, \al_{\pm}}(r) = n \pi - \al_{\pm} r^{-2} + O(r^{-6})\\
&Q_{n, 0}(r) = n \pi
}
where $Q_{n, \al_+}(r)$ is a trajectory on the branch of $W^{st}_n$ that increases to $n\pi$ as $r \to \infty$, and $Q_{n, \al_-}(r)$ is a  trajectory on the branch of $W^{st}_n$ that decreases to $n \pi$ as $r \to \infty$. Since the trajectory $Q_{n, \al_+}$ satisfies $Q_{n, \al_+}(r_0) = 0$ for some $r_0 >0$, we can obtain our solution $Q_n(r)$ to \eqref{hm} which satisfies~\eqref{bc} by rescaling $Q_{n, \al_+}(r)$ by $\la_0>0$, i.e., we set 
\ant{
Q_n(r) = Q_n^{+}( r / \la_0) = n \pi - \la_0^2 \al_+ r^{-2} + O(r^{-6})
}
where we note that $\la_0>0$ is uniquely chosen to ensure that the boundary condition $Q_n(1) = 0$ is satisfied.  Note that such rescalings amount to a translation in the $s$-variable above. Setting $\al_0 := \la_0^2 \al_+ $, the unique harmonic map $(Q_n(r), 0) \in \E_n$ therefore satisfies 
\EQ{\label{Q asym}
Q_n(r)  = n \pi - \al_0 r^{-2} + O(r^{-6})
}
as claimed above. 

\subsection{5d Reduction}

In the higher topological classes, $\E_n$ for $n \ge1$, we linearize about $Q=Q_n$ by writing
\begin{align*}
\psi = Q + \fy
\end{align*}
where $Q= Q_n$ is the unique harmonic map and energy minimizer in $\E_n$. If $\vec \psi \in \E_n$ is a wave map, then $\vec \fy\in \HH_n$ satisfies 
\EQ{\label{eq:fy nl}
&\fy_{tt}- \fy_{rr} - \frac{2}{r} \fy_r  +\frac{2 \cos(2Q)}{r^2} \fy = Z(r, \fy)\\ 
&Z(r, \fy):=\frac{\cos(2 Q)( 2 \fy -  \sin(2 \fy)) + 2 \sin(2Q)\sin^2(\fy)}{r^2}\\
&\fy(t, 1) = 0, \, \fy(t, \infty) = 0 \quad \forall t, \quad \vec \fy(0)=( \psi_0-Q, \psi_1)
}
The standard $5d$ reduction is given by setting $r u:= \fy$ and then  $\vec u$ solves 
\EQ{ \label{eq:u nl}
&u_{tt}- u_{rr} - \frac{4}{r}u_r +  V(r)u = F(r, u) + G(r, u), \quad r \ge 1\\ 
& u(t, 1) = 0 \quad \forall t,  \quad
 \vec u(0) = (u_0, u_1) \\
&V(r) :=\frac{2(\cos(2Q)-1)}{r^2}\\
&F(r, u) :=2\sin(2 Q) \frac{\sin^2( ru) }{r^3}\\
&G(r, u):=  \cos(2Q)\frac{(2ru - \sin(2ru))}{r^3}
}
We will consider radial initial data $(u_0, u_1) \in  \HH:=\dot{H}^1_0 \times L^2(\R^5_*)$ where $\R^5_* = \R^5 \setminus B(0, 1)$, 
\EQ{
\|(u_0, u_1)\|_{\HH}^2 := \int_1^{\infty} ((\p_ru_0(r))^2+ u_1^2(r)) \, r^4 \, dr
}
where $\dot{H}^1_0 (\R^5_*)$ is the completion under the first norm on the right-hand side above of all smooth radial compactly supported functions on $\{x \in \R^5 \mid \abs{x}>1\}$. 
We remark that the potential 
\EQ{\label{V def}
V(r):=\frac{2(\cos(2Q)-1)}{r^2}
} 
is real-valued, radial, bounded, smooth and by \eqref{Q asym} satisfies 
\EQ{\label{V decay}
V(r) = O(r^{-6}) \mas r \to \infty
}
Also, by \eqref{Q asym} we can deduce that 
\EQ{\label{FG est}
&\abs{F(r, u)} \lesssim  r^{-3}\abs{u}^2\\
&\abs{G(r, u)} \lesssim \abs{u}^3
}
For the remainder of the paper we deal exclusively with $u(t,r)$ in~$\R^5_*$ rather than
the equivariant wave map angle $\psi(t,r)$. In fact,  one can check that the Cauchy problem~\eqref{eq:psi wm} with data $(\psi_0, \psi_1) \in \E_n$ is equivalent to \eqref{eq:u nl}. To see this let $\vec \psi \in \E_n$ and set 
\EQ{
r\vec u (r) := ( \psi_0(r)- Q_n(r), \psi_1(r))
}
We claim that 
\EQ{
\| \vec \psi \|_{\HH_n} \simeq \|\vec u\|_{\HH}
}
Indeed, setting $\fy(r) := \psi_0(r)- Q_n(r)$ we see that 
\EQ{
 \int_1^\infty \fy_r^2(r) r^2\, dr \simeq \int_1^\infty u_r^2(r) r^4\, dr 
}
via Hardy's inequality and the relations
\[
\fy_r = r u_r + u = r u_r + \frac{\fy}{r}
\]
Therefore for each topological class $\E_n$ the map 
\EQ{\nn
 \vec \psi \mapsto \frac{1}{r}( \psi_0(r)- Q_n(r), \psi_1(r))
 }
 is an isomorphism between the spaces $\E_n$ and $\HH$ respectively. 

In particular, we will prove the analogous formulation
of Theorem~\ref{main} in the $u$-setting rather than the original one.  Scattering in this context
will mean that we approach a solution of~\eqref{eq:u nl} but with $V=F=G=0$. 

\section{Small Data Theory and Concentration Compactness}\label{cc}

\subsection{Global existence and scattering for data with small energy}
Here we give a brief review of the small data well-posedness theory for~\eqref{eq:u nl} that was developed in \cite{LS}.  As usual the small data theory rests on Strichartz estimates for the inhomogeneous linear, radial exterior wave equation with the potential $V$, 
\EQ{\label{eq:u lin}
&u_{tt}- u_{rr} - \frac{4}{r}u_r +  V(r)u = h\\
&u(t, 1) = 0 \quad \forall t \\
&\vec u(0) = (u_0, u_1) \in \HH
}
where $V(r)$ is as in \eqref{V def}.  We define $S_V(t)$ to be the exterior linear propagator associated to \eqref{eq:u lin}. 
The conserved energy associated to \eqref{eq:u lin} with $h=0$ is given by 
\begin{align*}
\E_L(u, u_t) = \frac{1}{2} \int_1^{\infty} (u_t^2 + u_r^2 + V(r)u^2) \, r^4 \,dr
\end{align*}
This energy has an important positive definiteness property: one has
\EQ{
\E_L(u, u_t) = \frac{1}{2} ( \|u_t\|_2^2 + \lan H u|u\ran),\quad H= -\Delta +V
}
It is shown in~\cite{Biz, LS} that $H$ is a nonnegative self-adjoint operator in $L^2(\R^5_*)$ (with a Dirichlet condition
at $r=1$), and moreover, that the threshold energy zero is regular; in other words, if $Hf=0$ where $f\in H^2\cap \dot H^1_0$
then $f=0$. It is now standard to conclude from this spectral information that for some constants $0<c<C$, 
\EQ{\label{EV}
c \|  f\|_{\dot H^1_0 }^2\le  \lan Hf|f\ran \le C\|  f\|_{\dot H^1_0 }^2 \quad \forall\;f\in  \dot H^1_0(\R^5_*) 
}
We sometimes write $\| \vec u\|_{\E}^2:=\E_L(\vec u)$, which satisfies 
\EQ{\label{norm comp}
\| \vec u\|_{\E} \simeq \| \vec u\|_{\HH} \quad \forall \vec u\in\HH
} 

In what follows we say a triple $(p, q, \gamma)$  is admissible if
\ant{
&p>2, \, \, q \ge 2\\
& \frac{1}{p} + \frac{5}{q} = \frac{5}{2}- \gamma\\
& \frac{1}{p} + \frac{2}{q} \le 1
}
For the free exterior $5d$ wave, i.e., the case $V=0$ in ~\eqref{eq:u lin}, Strichartz estimates were established in \cite{HMSSZ}. Although the estimates in  \cite{HMSSZ} hold in more general exterior settings, we state only the specific  example of these estimates that we need here. 
 \begin{prop}\cite{HMSSZ}\label{ext strich} Let $(p, q, \gamma)$ and $(r, s, \rho)$ be admissible triples. Then any solution  $\vec{v}(t)$ to 
 \EQ{\label{eq: inhom}
&v_{tt}- v_{rr} - \frac{4}{r}v_r = h\\
&\vec v(0)= (f, g) \in \HH(\R^{5}_*)\\
&v(t, 1) = 0 \quad \forall t \in \R
}
with radial initial data satisfies
\EQ{
\| \abs{\nabla}^{- \gamma} \nabla v\|_{L^p_tL^q_x} \lesssim \|(f, g) \|_{\HH} + \| \abs{\nabla}^{\rho} h\|_{L^{r'}_t L^{s'}_x}
}
\end{prop}

In \cite{LS}, the second and third authors showed that in fact the same family of  Strichartz estimates hold for \eqref{eq:u lin}.  
\begin{prop}\cite[Proposition~{$5.1$}]{LS}\label{strich} Let $(p, q, \gamma)$ and $(r, s, \rho)$ be admissible triples. Then any solution $\vec{u}(t)$ to \eqref{eq:u lin} with radial initial data satisfies 
\EQ{\label{strich est}
\| \abs{\nabla}^{- \gamma} \nabla u\|_{L^p_tL^q_x} \lesssim \|\vec u(0) \|_{\HH} + \| \abs{\nabla}^{\rho} h\|_{L^{r'}_t L^{s'}_x}
}
\end{prop}
With these Strichartz estimates the following small data, global well-posedness theory for~\eqref{eq:u nl} follows from the standard contraction argument. 

\begin{prop}\cite[Theorem~$1.2$]{LS} \label{small data}The exterior Cauchy problem for~\eqref{eq:u nl} is globally well-posed in $ \HH:=\dot H^1_{0}\times L^2(\R^5_*)$. Moreover, a solution $u$ scatters as $t\to\infty$ to a free wave, i.e., a solution $\vec u_L\in \HH$ of  
\EQ{\label{v free}
\Box u_L =0, \; r\ge1, \;  u_L(t, 1)=0, \; \forall t\ge0
}
if and only if $\|u\|_{\calS}<\infty$ where $\calS=L^3_t([0,\infty); L^6_x(\R^5_*))$. In particular, there exists a constant $\delta>0$ small 
so that if $\|\vec u(0)\|_{\HH}<\delta$, then $u$ scatters to free waves as $t \to\pm\infty$. 
\end{prop}

\begin{rem} 
We remark that in \cite[Theorem $1.2$]{LS}, the conclusions of Proposition~\ref{small data} were phrased in terms of the original wave map angle $\psi$ where here the result is phased in terms of $u(t, r) := \frac{1}{r} (\psi(t, r) - Q_n(r))$. As we saw in Section~\ref{prelim} this passage to the $u-$formulation  is allowed since the map $\vec u = \frac{1}{r} ( \psi-Q_n, \psi_t)$ is an isomorphism between the energy class $\E_n$  and $\HH:= \dot{H_0}^1 \times L^2(\R^5_*)$, respectively. 
\end{rem}
We refer the reader to \cite{LS} for the details regarding Proposition~\ref{strich} and Proposition~\ref{small data}. For convenience, we recall how the scattering norm $L^3_tL^6_x$ is obtained. By Proposition~\ref{strich}, solutions to \eqref{eq:u lin} satisfy
\EQ{\label{33strich}
\|   u\|_{L^3_t(\R; \dot{W}^{\frac{1}{2}, 3}_x(\R^5_*))} \lesssim \|\vec u(0) \|_{\HH} + \| h\|_{L^{1}_t L^{2}_x + L^{\frac{3}{2}}_tL^{\frac{30}{17}}_x}
}
As in \cite{LS}, we claim the embedding $\dot W_x^{\frac{1}{2}, 3} \hookrightarrow L^6_x$ for radial functions in $r\ge1$ in $\R^5_*$. Indeed, one checks 
via the fundamental theorem of calculus that $\dot W_x^{ 1, 3} \hookrightarrow L^\infty_x$. More precisely, 
\EQ{
|f(r)|\le r^{-\frac{2}{3}} \| f\|_{\dot W_x^{1, 3} }
}
Interpolating this with the  embedding $L^3\hookrightarrow L^3$ we obtain the claim. 
From~\eqref{33strich} we infer the weaker Strichartz estimate
\EQ{\label{36 strich}
\|   u\|_{L^3_t(\R; L^6_x(\R^5_*))} \lesssim \|\vec u(0) \|_{\HH} + \| h\|_{L^{1}_t (\R; L^{2}_x( \R^5_*)) + L^{\frac{3}{2}}_t( \R; L^{\frac{30}{17}}_x(\R^5_*))}
}
 which suffices for our purposes. Indeed, using \eqref{36 strich} on the nonlinear equation~\eqref{eq:u nl} gives 
 \EQ{\nn
\|   u\|_{L^3_t(\R; L^6_x(\R^5_*))} &\lesssim \|\vec u(0) \|_{\HH} + \| F(r, u)+ G(r, u)\|_{L^{1}_t L^{2}_x + L^{\frac{3}{2}}_tL^{\frac{30}{17}}_x}\\
& \lesssim \|\vec u(0) \|_{\HH} + \| r^{-3}u^2\|_{L^{\frac{3}{2}}_tL^{\frac{30}{17}}_x} +  \| u^3\|_{L^{1}_t L^{2}_x}\\
& \lesssim \|\vec u(0) \|_{\HH} + \|r^{-3}\|_{L^{\infty}_tL^{\frac{30}{7}}_x}\| u^2\|_{L^{\frac{3}{2}}_tL^{3}_x} +  \| u\|_{L^{3}_t L^{6}_x}^3\\
& \lesssim \|\vec u(0) \|_{\HH} +\| u\|_{L^{3}_tL^{6}_x}^2 +  \| u\|_{L^{3}_t L^{6}_x}^3
}
where we have estimated the size of the nonlinearity $h=F(r, u) + G(r, u)$ using \eqref{FG est}. Thus for small initial data, $\|\vec u(0) \|_{\HH} <\delta$, we obtain the global a priori estimate 
\EQ{\label{36 strich nl} 
\|   u\|_{L^3_t(\R; L^6_x(\R^5_*))}  \lesssim \|\vec u(0) \|_{\HH}  \lesssim \de
}
from which the small data scattering statement in Proposition~\ref{small data} follows.

\subsection{Concentration Compactness}

We now formulate the concentration compactness principle relative to the linear
wave equation with a potential, see~\eqref{eq:u lin} with $h=0$. This is what we mean by ``free" in Lemma~\ref{lem:BG}. Note
that this is a different meaning of ``free" than the one used in
Proposition~\ref{small data}. However, observe that any solution to \eqref{eq:u lin} with $h=0$, which
is in $L^3_tL^6_x$ must scatter to ``free" waves, where ``free" is in the sense of
Proposition~\ref{small data}.

\begin{lem}\label{lem:BG}
Let $\{u_n\}$ be a sequence of free radial waves bounded in $\HH=\dot H^1_{0}\times L^2(\R^5_*)$. Then after replacing it by a subsequence, 
there exist a sequence of free solutions $v^j$ bounded in $\HH$, and sequences of times $t_n^j\in\R$ such that for  $\ga_n^k$ defined by
\EQ{\label{eq:BGdecomp}
  u_n(t) = \sum_{1\le j<k} v^j(t+t_n^j) + \ga_n^k(t) }
we have for any $j<k$, 
\EQ{\label{gamma weak}
\vec\ga_n^k(-t_n^j) \rightharpoonup 0
}
 weakly in $\HH$ as $n\to\I$,  as well as 
\EQ{\label{eq:tdiverge}
 \lim_{n\to\I} |t_n^j-t_n^k| = \I     
}
and the errors $\ga_{n}^{k}$ vanish asymptotically
in the sense that
\EQ{ \label{gammavanish}
 \pt \lim_{k\to \I} \limsup_{n\to\I} \|\ga_n^k\|_{(L^\I_tL^p_x\cap L^3_t L^6_x)(\R\times\R^5_*)}=0 \quad \forall \; \frac{10}{3}<p<\infty 
}
Finally, one has orthogonality of the free energy with a potential, cf.~\eqref{norm comp}, 
\EQ{ \label{H1 orth}
 \| \vec u_n \|_{\E}^2 = \sum_{1\le j<k} \|\vec v^j\|_{\E}^2 + \|\vec\ga_n^k\|_{\E}^2 +o(1)
} 
as $n\to\I$. 
\end{lem}

The proof is essentially identical with that of Lemma~3.2 in \cite{LS}. In fact, instead of the Strichartz estimates for $\Box$ in $\R^5_*$
we use those from Proposition~\ref{strich} above.

Applying this decomposition to the nonlinear equation requires a perturbation lemma which we now formulate. All spatial norms are 
understood to be on~$\R^5_*$. The exterior propagator $S_V(t)$ is as above. 

\begin{lem} \label{PerturLem} 
There are continuous functions $\eps_0,C_0:(0,\I)\to(0,\I)$ such that the following holds: 
Let $I\subset \R$ be an open interval (possibly unbounded), $u,v\in C(I;\dot H_{0}^{1})\cap C^{1}(I;L^{2})$ radial functions satisfying for some $A>0$ 
\EQ{\nn 
\|\vec u\|_{L^\infty(I;\HH)} +  \|\vec v\|_{L^\infty(I;\HH)} +   \|v\|_{L^3_t(I;L^6_x)} & \le A \\   
 \|\glei(u)\|_{L^1_t(I;L^2_x)} 
   + \|\glei(v)\|_{L^1_t(I;L^2_x)} + \|w_0\|_{L^3_t(I;L^6_x)} &\le \eps \le \eps_0(A),}
where $\glei(u):=(\Box+V)u-F(r, u)-G(r, u)$ in the sense of distributions, and $\vec w_0(t):=S_{V}(t-t_0)(\vec u-\vec v)(t_0)$ with $t_0\in I$ arbitrary but fixed.  Then
\EQ{ \nn 
  \|\vec u-\vec v-\vec w_0\|_{L^\I_t(I;\HH)}+\|u-v\|_{L^3_t(I;L^6_x)} \le C_0(A)\eps.} 
  In particular,  $\|u\|_{L^3_t(I;L^6_x)}<\I$. 
\end{lem}

The proof of this lemma is essentially identical with that of \cite[Lemma 3.3]{LS}. 
The only difference is that we use the propagator $S_V$ instead of $S_0$.

\subsection{Critical Element} 

We now turn to the proof of Theorem~\ref{main} following the concentration compactness 
methodology from~\cite{KM06, KM08}. We begin by noting that Theorem~\ref{main} was proved in
the regime of all 
energies slightly above the ground state energy~$\E(Q_n,0)$ in~\cite[Theorem 1.2]{LS}, see also
Proposition~\ref{small data} above. As usual, we assume that Theorem~\ref{main} fails and construct
a {\em critical element} which is a non-scattering solution of minimal energy, $E_*$, which is necessarily strictly
bigger than~$\E(Q_n,0)$. This is done in the following proposition on the level of the semi-linear formulation
given by~\eqref{eq:u nl}. 

\begin{prop}
\label{prop:3}
Suppose that Theorem~\ref{main} fails. 
Then there exists a  nonzero energy solution to~\eqref{eq:u nl} (referred to as a critical element) $\vec u_*(t)$ for $t\in \R$ with the property
that the trajectory 
\EQ{\label{compK}
\calK:= \{\vec u_*(t) \mid t\in \R\}
} 
is pre-compact in $\HH(\R^5_*)$. 
\end{prop}
\begin{proof} 
Suppose that the theorem fails. Then 
there exists a bounded sequence of $\vec \psi_j=( \psi_{0, j}, \psi_{1, j}) \in \E_n$ 
with 
\EQ{
\E( \vec \psi_j) \to E_* >0
}
and a bounded sequence $\vec u_j:=(u_{0,j},u_{1,j})\in \HH$ where $\vec u_j(r)= \frac{1}{r}(\vec \psi_j(r)- (Q(r), 0)) $ with 
\[
 \|u_j\|_{\calS}\to\infty
\]
where $u_n$ denotes the global evolution of $\vec u_n$ of~\eqref{eq:u nl}. We may assume that $E_*$ is minimal with this property. 
Applying Lemma~\ref{lem:BG} to the free evolutions $S_V$ of $\vec u_j(0)$ 
yields  free waves $v^{i}$ and times $t^{i}_{j}$ as in~\eqref{eq:BGdecomp}.  
Let $U^{i}$ be the nonlinear profiles of $(v^{i}, t^{i}_{j})$, i.e., those energy solutions of~\eqref{eq:u nl} which satisfy
\EQ{\nn
\lim_{t\to t^{i}_{\I} }\| \vec v^{i}(t) - \vec U^{i}(t)\|_{\HH} \to0
}
where $\lim_{j\to\I} t^{i}_{j} = t^{i}_{\I}\in [-\I,\I]$.   The $U^{i}$ exist locally around $t=t^{i}_{\I}$ 
by the local existence and scattering theory, see Proposition~\ref{small data}.  Note that here and throughout we are using the equivalence of norms in \eqref{norm comp}.
Locally around $t=0$ one has the following {\em nonlinear profile decomposition}
\EQ{ \label{eq:nonlinearprofile} 
u_{j} (t) =  \sum_{i<k}  U^{i}(t+t^{i}_{j}) + \ga_{j}^{k}(t) + \eta_{j}^{k}(t)
}
where $\| \vec \eta_{j}^{k}(0)\|_{\HH}\to0$ as $j\to\I$.  
Now suppose that either there are two non-vanishing $v^{j}$, say
$v^{1}, v^{2}$, or that 
\EQ{\label{eq:ganonvanish}
\limsup_{k\to\I}\limsup_{j\to\I} \| \vec \ga^{k}_{j}\|_{\E}>0
}
Note that the left-hand side does not depend on time since $\ga_{j}^{k}$ is a free wave. 
By the minimality of $E_{*}$ and the orthogonality of the nonlinear energy--which follows from \eqref{eq:tdiverge} and \eqref{gamma weak}--each $U^{i}$ is a global solution and scatters with $\| U^{i}\|_{L^{3}_t L^{6}_{x}} <\I$.   

We now apply Lemma~\ref{PerturLem} on $I=\R$ with $u=u_{j}$ and 
\EQ{ \label{eq:sumUj}
v(t)=\sum_{i<k} U^{i}(t+t^{i}_{j})
}
That $\| \glei(v)\|_{L^{1}_{t} L^{2}_{x}}$ is small for large $n$ follows from~\eqref{eq:tdiverge}. 
To see this, note that with $N(v):=F(r, v)+G(r, v)$, 
\EQ{\nn
\glei(v) &= (\Box+V) v-F(r, v)-G(r, v) \\
& = \sum_{i<k} N(U^{i}(t+t^{i}_{j}))    -N\big( \sum_{i<k} U^{i}(t+t^{i}_{j}) \big)
}
The difference on the right-hand side here only consists of terms which involve at least one pair of distinct $i,i'$. 
But then $\| \glei(v)\|_{L^{1}_{t} L^{2}_{x}}\to 0$ as $j\to\I$ by~\eqref{eq:tdiverge}.  
In order to apply Lemma~\ref{PerturLem}   it is essential  that 
\EQ{\label{eq:kunif}
\limsup_{j\to\I} \big\| \sum_{i<k} U^{i}(t+t^{i}_{j})  \big\|_{L^{3}_t L^{6}_{x}} \le A < \I
}
{\em uniformly in} $k$, which follows from \eqref{eq:tdiverge}, \eqref{H1 orth}, and Proposition~\ref{small data}. 
The point here is that the sum can be split into one over $1\le i<i_{0}$ and another over $i_{0}\le i<k$.
This splitting is performed in terms of the energy, with $i_{0}$ being chosen such that for all $k>i_{0}$
\EQ{\label{eq:Ujeps0}
\limsup_{j\to\I}\sum_{ i_{0} \le i < k}\| \vec U^{i}(t^{i}_{j})\|_{\HH}^{2} \le \eps_{0}^{2}
}
where $\e_{0}$ is fixed such that the small data result of Proposition~\ref{small data} applies. 
Clearly, \eqref{eq:Ujeps0}  follows from~\eqref{H1 orth}. Using~\eqref{eq:tdiverge} as well
as the small data scattering theory   one now obtains 
\EQ{\label{eq:L3 orth}
\limsup_{j\to\I}\Big\|  \sum_{ i_{0} \le i< k}  U^{i}(\cdot+t^{i}_{j})  \Big\|_{L^{3}_{t}L^{6}_{x}}^{3} &=\sum_{ i_{0} \le i < k}  \big\|   
U^{i}(\cdot)  \big\|_{L^{3}_{t}L^{6}_{x}}^{3} \\ 
& \le C \limsup_{j\to\I} \Big( \sum_{ i_{0} \le i < k} \| \vec U^{i}(t^{i}_{j})\|_{\HH}^{2}  \Big)^{\f32}
}
with an absolute constant $C$. This implies \eqref{eq:kunif}, uniformly in~$k$. 

Hence one can take $k$ and $j$  so large that Lemma~\ref{PerturLem}  applies to~\eqref{eq:nonlinearprofile} whence
\EQ{\nn
\limsup_{j\to\I} \|u_{j}\|_{L^{3}_t L^{6}_{x}} < \I
}
which is a contradiction. Thus, there can be only one nonvanishing $v^{i}$, say $v^{1}$, and  moreover 
\EQ{\label{eq:gavanish}
 \limsup_{j\to\I} \| \vec \ga^{2}_{j}\|_{\HH} = 0
}
Thus, if we let $\vec \psi^1 $ be the wave map angle associated to $\vec U^{1}$ then we have $\E(\vec \psi^1)= E_*$.  By the preceding, necessarily 
\EQ{\label{eq:U1Stinf}
\|U^1\|_{L^{3}_t L^{6}_{x}} = \I
}
Therefore, $U^1=:u_*$ is the desired critical element. Suppose that 
\EQ{\label{eq:U1Stinf+}
\|u_*\|_{L^{3}_t([0,\I); L^{6}_{x})} = \I
}
Then we claim that $$\calK_+:= \{\vec u_*(t) \mid t\ge0\} $$ is precompact in~$\HH$. If not, then there exists $\delta>0$ so that 
for some infinite sequence $t_n\to\I$ one has
\EQ{\label{eq:U1delta}
\| \vec u_*(t_n)- \vec u_*(t_m)\|_\HH>\delta \quad \forall \; n>m
}
Applying Lemma~\ref{lem:BG} to $U^1(t_n)$ one concludes via the same argument as before
based on the minimality of $E_*$ and \eqref{eq:U1Stinf} that 
\EQ{\label{eq:U1BG}
\vec u_*(t_n) = \vec v(\tau_n) + \vec \gamma_n(0)
}
where $\vec v$, $\vec \gamma_n$ are free   waves in~$\HH$, and $\tau_n$ is some
sequence in~$\R$. Moreover, $\| \vec\gamma_n\|_\HH\to0$ as $n\to\I$. 
If $\tau_n\to \tau_\I\in\R$, then \eqref{eq:U1BG} and~\eqref{eq:U1delta} lead to a contradiction.
If $\tau_n\to\I$, then 
\EQ{\nn
\| v(\cdot+\tau_n)\|_{L^{3}_t([0,\I); L^{6}_{x})}  \to0 \qquad \text{\ as\ }n\to\I
}
implies via the local wellposedness theory
that $\| u_*(\cdot+t_n)\|_{L^{3}_t([0,\I); L^{6}_{x})}<\I$ for all large $n$, which is a contradiction to~\eqref{eq:U1Stinf+}.
If $\tau_n\to-\I$, then  
\EQ{\nn
\| v(\cdot+\tau_n)\|_{L^{3}_t( (-\I,0]; L^{6}_{x})}  \to0 \qquad \text{\ as\ } n\to\I
}
implies that $\| u_*(\cdot+t_n)\|_{L^{3}_t( (-\I,0] ; L^{6}_{x})}<C<\I$ for all large $n$ where $C$ is some fixed constant. 
Passing to the limit yields a contradiction to~\eqref{eq:U1Stinf} and~\eqref{eq:U1delta} is seen to be false, concluding the proof of compactness
of~$\calK_+$. 

Finally, we need to make sure that $u_*(t)$ is precompact with respect to both $t\to+\infty$ and $t\to-\infty$, see~\eqref{compK}.
To achieve the latter, we extract another critical element from the sequence 
\[
\{ \vec u_*(n)\}_{n=1}^\infty \subset \HH
\]
Indeed, by the compactness that we have already established we can pass to a strong limit $\vec u_n\to \vec u_\infty$ in~$\HH$,
which has the same energy~$E_*$. 
By construction, the nonlinear evolution~\eqref{eq:u nl} with data~$\vec u_\infty$ has infinite $L^3_t L^6_x$-norm in both time directions.
Therefore, the same compactness argument as above   concludes the proof. Indeed, the solution given by $\vec u_\infty$ is now our
desired critical element. 
\end{proof}

In Section~\ref{sec:rigid} we will show that $u_*$ cannot exist. In order to do so, we need to develop
another tool for the linear evolution.

\section{The linear external energy estimates in $\R^5$}

We now turn to our main new ingredient from the linear theory, which is Proposition~\ref{linear prop}.
In order to motivate this result, we first review the analogous statements in dimensions $d=1$ and $d=3$.

Suppose $w_{tt}-w_{xx}=0$ with smooth energy data $(w(0),\dot w(0))=(f,g)$. Then by local energy conservation
\[
\int_{x>a} \f12(w_t^2 + w_x^2)(0,x)\, dx -\int_{x>T+a} \f12(w_t^2 + w_x^2)(T,x)\, dx  = \f12 \int_0^T(w_t+w_x)^2(t,t+a)\, dt
\]
for any $T>0$ and $a\in\R$. Since $(\p_t-\p_x)(w_t+w_x)=0$, we have that 
\EQ{\nn
&\f12 \int_0^T(w_t+w_x)^2(t,t+a)\, dt = \f12 \int_0^T(w_t+w_x)^2(0,a+2t)\, dt  \\
&= \f14 \int_a^{a+2T}(w_t+w_x)^2(0,x)\, dx = \f14 \int_a^{a+2T}(f_x+g)^2(x)\, dx
}
Consequently, 
\EQ{\nn
& \int_{x>a} \f12(w_t^2 + w_x^2)(0,x)\, dx- \lim_{T\to\I} \int_{x>T+a} \f12(w_t^2 + w_x^2)(T,x)\, dx  \\
&=\f14 \int_a^{\I}(f_x+g)^2(x)\, dx
}
and thus
\EQ{\nn
&\min_{\pm}  \Big[ \int_{x>a} \f12(f_x^2 + g^2)(0,x)\, dx- \lim_{T\to\pm \I} \int_{x>|T|+a} \f12(w_t^2 + w_x^2)(T,x)\, dx \Big] \\
&\le \f14 \int_a^{\I}(f_x^2+g^2)(x)\, dx
}
whence 
\EQ{\label{EE1}
\max_{\pm}\lim_{T\to\pm \I} \int_{x>|T|+a} \f12(w_t^2 + w_x^2)(T,x)\, dx \ge \f14 \int_a^{\I}(f_x^2+g^2)(x)\, dx
}
Here we used that $t\mapsto -t$ leaves $f$  unchanged, but turns $g$ into~$-g$. 

Given $\Box u=0$ radial in three dimensions,   $w(t,r)=ru(t,r)$ solves $w_{tt}-w_{rr}=0$. Consequently, 
\eqref{EE1} gives the following estimate from~\cite[Lemma 4.2]{DKM1}, see also \cite{DKM3, DKM4, DKM5}: for any $a\ge 0$ one has
\EQ{\label{3d extE}
&\max_{\pm}  \lim_{T\to\pm \I} \int_{r>|T|+a} \f12( (ru)_r^2+(ru_t)^2 )(T,r)\,  dr\\ & \ge \f14 \int_{r>a} ((rf)_r^2+(rg)^2)(r)\,   dr
}
where $u(0)=f$, $\dot u(0)=g$. The left-hand side of~\eqref{3d extE} equals
\EQ{\label{EE}
\max_{\pm}  \lim_{T\to\pm \I} \int_{r>|T|+a} \f12( u_r^2+u_t^2 )(T,r)\, r^2 dr
}
by the standard dispersive properties of the wave equation. 
The right-hand side, on the other hand,  exhibits the following dichotomy: if $a=0$, then it equals half of the full energy
\[
\f14 \int_0^\I (f_r^2+g^2)(r)\,r^2   dr
\]
However, if $a>0$, then integration by parts shows that it equals  (ignoring the constant from the spherical measure in $\R^3$)
\[
\f14 \int_{r>a} (f_r^2+g^2)(r)r^2\,   dr-\f14 af^2(a)=\f14 \| \pi_a^\perp (f,g)\|_{\dot H^1 \times L^2(r>a)}^2
\]
where  $\pi_a^\perp = \Id-\pi_a$ and $\pi_a$ is the orthogonal projection onto the line $$\{ (cr^{-1},0)\mid c\in\R\}\subset \dot H^1\times L^2(r>a).$$ 
The appearance of this projection is natural, in view of the fact that the Newton potential $r^{-1}$ in~$\R^3$ yields an explicit solution
to $\Box u=0$, $u(0,r)=r^{-1}, \dot u(0,r)=0$: indeed, one has $u(r,t)=r^{-1}$ in $r>|t|+a$ for which~\eqref{EE}
vanishes. Since $r^{-1}\not\in L^2(r>1)$ no projection appears in the time component. 
In contrast, the Newton potential in $\R^5$, viz.~$r^{-3}$, does lie in $H^1(r>a)$ for any $a>0$. This explains why in $\R^5$ we need
to project away from a plane rather than a line, see~\eqref{R5ext} and the end of the proof of
Proposition~\ref{linear prop}.

\begin{prop}\label{linear prop}
Let $\Box u=0$ in $\R^{1+5}_{t,x}$ with radial data $(f,g)\in \dot H^1\times L^2(\R^5)$.
Then with some absolute constant $c>0$ one has for every $a>0$
\EQ{
\label{R5ext}
\max_{\pm}\;\limsup_{t\to\pm\I} \int_{r>a+|t|}^\I(u_t^2 + u_r^2)(t,r) r^4\, dr \ge c \| \pi_a^\perp (f,g)\|_{\dot H^1\times L^2(r>a)}^2
}
where $\pi_a=\Id-\pi_a^\perp$ is the orthogonal projection onto the plane $$\{( c_1 r^{-3}, c_2 r^{-3})\:|\: c_1, c_2\in\R\}$$
in the space $\dot H^1\times L^2(r>a)$.  The left-hand side of~\eqref{R5ext} vanishes for all data in this plane. 
\end{prop}
\begin{rem} We note that by finite propagation speed Proposition~\ref{linear prop} with $a >1$ holds as well for solutions $v(t)$ to the free radial wave equation in $\R \times \R^{5}_*$ with a Dirichlet boundary condition at $r=1$. 
\EQ{\label{eq:free}
&v_{tt}- v_{rr} - \frac{4}{r}v_r = 0\\
&\vec v(0)= (f, g)\\
&v(t, 1) = 0 \quad \forall t \in \R
}
\end{rem}

\begin{proof}
By the basic energy estimate we may assume that $f,g$ are compactly supported and smooth, say. We first note
that it suffices to deal with data $(f,0)$ and $(0,g)$ separately. Indeed, reversing the time direction keeps the former
fixed, whereas the latter changes to $(0,-g)$. This implies that we may choose the time-direction so as
to render the bilinear interaction term between the two respective solutions nonnegative on the left-hand side of~\eqref{R5ext}. 

We begin with data $(f,0)$ and set $w(t,r):= r^{-1}(r^3 u(t,r))_r$, see~\cite{KM11}. Throughout this proof, the singularity at $r=0$
plays no role due to the fact that $r\ge a+|t|\ge a>0$. Then   
\EQ{\nn
w_{tt}-w_{rr} &= r^2\p_r\big( u_{tt} - u_{rr} - \f{4}{r}u_r\big) + 3r\big(u_{tt}-u_{rr}- \f{4}{r}u_r\big)  =0
}
From the d'Alembert formula, 
\[
\limsup_{t\to \I} \int_{a+t}^\I w^2(t,r) \, dr \ge \f14 \int_a^\I w^2(0,r)\, dr
\]
which is the same as
\EQ{
\label{f0}
&\limsup_{t\to\I} \int_{a+t}^\I (r^2 u_r(t,r)+3r u(t,r))^2 \, dr \ge 
 \f14 \int_a^\I (r^2 f'(r)+3r f(r))^2    \, dr
}
By our assumption on the data, we have the point wise bound
\[
|u(t,r)|\le Ct^{-2} \chi_{[R-t\le r\le R+t]}
\]
for $t\ge1$ and some large $R$.  Hence, \eqref{f0} equals
\EQ{
\label{f0*}
&\limsup_{t\to\I} \int_{a+t}^\I u_r^2(t,r) r^4 \, dr \ge 
 \f14\Big( \int_a^\I r^4 f'(r)^2    \, dr - 3a^3 f(a)^2\Big)
}
where we integrated by parts on the right-hand side.  Finally, one checks that
\[
\tilde f(r):= f(r) - \f{a^3}{r^3} f(a)
\]
is the orthogonal projection perpendicular to $r^{-3}$ in $\dot H^1(r>a)$ in $\R^5$ and that it 
satisfies
\[
\int_a^\I r^4 \tilde f'(r)^2\, dr=\int_a^\I r^4 f'(r)^2    \, dr - 3a^3 f(a)^2
\]
which agrees with the right-hand side of~\eqref{f0*} and concludes the proof of \eqref{R5ext} for data $(f,0)$. 

For data $(0,g)$ we use the new  dependent variable
\EQ{\label{dep2}
v(t,r):= \int_r^\I s\p_t u(t,s)\, ds
}
By direct differentiation and integration by parts one verifies that $v$ solves the $3$-dimensional radial wave equation, viz.
\[
v_{tt} - v_{rr} - \f{2}{r} v_r=0
\]
Moreover, $v_t(0,r)=0$. 
From the exterior energy estimate in $\dim=3$, i.e.,  \eqref{3d extE}, 
\EQ{\label{3d}
\limsup_{t\to \I} \int_{a+t}^\I((rv)_t^2 + (rv)_r^2)(t,r) \, dr \ge
\f12 \int_{a}^\I((rv)_t^2 + (rv)_r^2)(0,r) \, dr
}
where we have used the fact that for data $(v_0, 0)$ or $(0, v_1)$ the estimate \eqref{3d extE} holds in both time directions. 
By our assumption on the data and stationary phase
\[
|v(t,r)|\le Ct^{-1}\chi_{[r\le R+t]}, \quad |v_t(t,r)|\le Ct^{-2}\chi_{[r\le R+t]}
\]
Hence~\eqref{3d} reduces to
\EQ{\label{3d*}
\limsup_{t\to \I} \int_{a+t}^\I v_r^2(t,r) r^2\, dr \ge
\f12 \int_{a}^\I  (rh'(r)+h(r))^2 \, dr
}
where $h(r):=\int_r^\I s g(s)\, ds$.  Inserting~\eqref{dep2} on the left-hand side and
integrating by parts on the right-hand side yields
\EQ{
\limsup_{t\to \I} \int_{a+t}^\I 2u_t^2(t,r) r^4\, dr &\ge
 \int_{a}^\I  h'(r)^2 r^2\, dr - ah^2(a)  \\
& =  \int_{a}^\I  g(r)^2 r^4\, dr - a \Big( \int_a^\I \rho g(\rho)\, d\rho\Big)^2
}
Finally, the right-hand side here is $\| \tilde g\|_{L^2(r>a)}^2$ where
\[
\tilde g(r) := g(r) - a r^{-3} \int_a^\I sg(s)\, ds 
\]
is the orthogonal projection perpendicular to $r^{-3}$ in $L^2(r>a)$ in $\R^5$.

For data $(r^{-3},0)$ the solution equals $r^{-3}$ on $r>t+a\ge a>0$ since $r^{-3}$ is the Newton potential in~$\R^5$.
Similarly, data $(0,r^{-3})$ produce the solution $tr^{-3}$ on the same region.  In both cases, the left-hand side of~\eqref{R5ext}
vanishes. 
\end{proof}

\section{Rigidity Argument}
\label{sec:rigid}

In this section we will complete the proof of Thereom~\ref{main} by showing that a critical element as constructed in Section~\ref{cc} does not exist. In particular, we prove the following proposition:
\begin{prop}[Rigidity Property]\label{rigidity}
Let $\vec u(t) \in \HH:= \dot{H}^1_0 \times L^2( \R^5_*)$ be a global solution to \eqref{eq:u nl} and suppose that the trajectory 
\ant{
K := \{\vec u(t) \mid t \in \R\}
}
is pre-compact in $\HH$. Then $\vec u(t) \equiv (0,0)$. 
\end{prop}
First note that the pre-compactness of $K$ immediately implies that the energy of $\vec u(t)$ on the exterior cone $\{r \ge R+ \abs{t}\}$ vanishes as $\abs{t} \to \infty$.
\begin{cor}\label{compact} Let $\vec u(t)$ be as in Proposition~\ref{rigidity}. Then for any $R\ge1$ we have 
\ali{\label{nl ext}
\| \vec u(t) \|_{\HH(r \ge R+ \abs{t})} \to 0 \mas \abs{t} \to \infty. 
}
\end{cor}
The proof of Proposition~\ref{rigidity} will proceed in several steps. The rough outline is to first use Corollary~\ref{compact} together with Proposition~\ref{linear prop} to determine the precise asymptotic behavior of $u_0(r)=u(0, r)$ and $u_1(r) = u_t(0,  r)$ as  $r \to \infty$. Namely, we show that
\EQ{\label{asymp}
&r^3 u_0( r) = \ell_o + O(r^{-3}) \mas r \to \infty\\
&r \int_r^{\infty} u_1( \rho) \rho \, d\,\rho = O(r^{-1}) \mas r \to \infty
}
 We will then argue by contradiction to show that $\vec u(t, r) = (0, 0)$ is the only possible solution that has both a pre-compact trajectory and initial data satisfying \eqref{asymp}.

\subsection{Step 1}
We use the exterior estimates for the free radial wave equation in Proposition~\ref{linear prop} together with Corollary~\ref{compact}   to deduce the following inequality for the pre-compact trajectory $\vec u(t)$. 
\begin{lem}\label{key ineq lem}
There exists $R_0>1$ such that for every $R \ge R_0$ and for all $t \in \R$ we have 
\EQ{\label{key ineq}
\| \pi^{\perp}_{R}\, \vec u(t)\|_{\HH(r \ge R)}^2 &\lesssim R^{-22/3}\| \pi_{R}\, \vec u(t)\|_{\HH(r \ge R)}^2 \\
&\quad + R^{-11/3}\| \pi_{R}\, \vec u(t)\|_{\HH(r \ge R)}^4 +\| \pi_{R}\, \vec u(t)\|_{\HH(r \ge R)}^6 
} 
where again $P(R):=\{ (k_1 r^{-3}, k_2 r^{-3})\mid k_1, k_2\in\R, \, r>R\}$,  $\pi_{R}\,$ denotes the orthogonal projection onto $P(R)$ and $\pi_{R}\,^\perp$ denotes the orthogonal projection onto the 
orthogonal complement of the plane $P(R)$
in $\HH(r>R;\R^5_*)$.  We note that~\eqref{key ineq} holds with a constant that is uniform in $t \in \R$. 
\end{lem}

In order to prove Lemma~\ref{key ineq lem}  we need a preliminary result concerning the nonlinear evolution for a modified Cauchy problem which is adapted to capture the behavior of our solution $\vec u(t)$ only on the exterior cone $\{(t, r) \mid r \ge R+ \abs{t}\}$. Since we will only consider the evolution  -- and in particular the vanishing property \eqref{nl ext} -- on the exterior cone we can, by finite propagation speed, alter the nonlinearity and the potential term in \eqref{eq:u nl} on the interior cone $\{1 \le r \le R + \abs{t}\}$ without affecting the flow on the exterior cone.  In particular, we can make the potential and the nonlinearity small on the interior of the cone so that for small initial data we can treat the  potential and nonlinearity as small perturbations.  
 
 With this in mind, for every $R>1$ we define $Q_R(t, r)$ by setting
\EQ{ 
Q_R(t, r):= \begin{cases} Q(R+ \abs{t}) \mfor 1 \le r \le R + \abs{t}\\ Q(r) \mfor r \ge R+ \abs{t}\end{cases}
}
Next, set  
\begin{align*} 
&V_R(t, r):= \begin{cases} 2(R+ \abs{t})^{-2} (\cos(2Q_R(t, r)) - 1)\mfor 1 \le r \le R+ \abs{t} \\  2r^{-2}(\cos(2Q(r)) - 1) \mfor r \ge R+ \abs{t} \end{cases}\\ 
&F_{R}(t, r, h) :=  \begin{cases} 2(R+ \abs{t})^{-3}\sin(2Q_R(t, r)) \sin^2((R+\abs{t})h) \mfor 1 \le r \le R+ \abs{t} \\ 2r^{-3}\sin(2Q( r)) \sin^2(r h) \mfor r \ge R+ \abs{t} \end{cases}\\ 
&G(r, h):= r^{-3} \cos(2Q(r))(2rh - \sin(2rh)) \quad \forall \, r\ge 1
\end{align*}
Note that for $R$ large enough we have, using \eqref{Q asym} and \eqref{V decay} that 
\begin{align}
&\abs{V_R(t, r)} \lesssim \begin{cases} (R+\abs{t})^{-6} \mfor 1 \le r \le R + \abs{t}\\ r^{-6} \mfor r \ge R + \abs{t}\end{cases}\label{VR decay}
\\
&\abs{F_R(t, r, h)}  \lesssim \begin{cases} (R+\abs{t})^{-3} \abs{h(t, r)}^2 \mfor 1 \le r \le R + \abs{t}\\ r^{-3} \abs{h(t,r)}^2 \mfor r \ge R + \abs{t}\end{cases}\label{FR decay}\\
&\abs{G(r, h)} \lesssim \abs{h(t, r)}^3 \mfor r \ge 1, \quad \forall t \in \R \label{G decay}
\end{align}
We consider the modified Cauchy problem in $\R \times \R^5_*$: 
\EQ{\label{eq:h}
&h_{tt}- h_{rr} - \frac{4}{r} h_r =\NN_R(t, r,h) \\
&\NN_R(t, r, h) :=-V_R(t, r) h + F_R(t, r, h) + G(r, h)\\
&h(1, t) = 0 \quad\forall t \in \R \\
& \vec h(0) = (h_0, h_1) \in \HH
}
\begin{lem}\label{m cp} There exists $R_0>0$ and there exists $\de_0>0$ small enough so that for all $R>R_0$ and all  initial data $\vec h(0) = (h_0, h_1) \in \HH$ with 
\ant{
\|\vec h(0) \|_{\HH}^2  \le \de_0
} 
there exists a unique global solution $\vec h(t) \in \HH$ to \eqref{eq:h}. In addition $\vec h(t) $ satisfies 
\ali{ \label{a priori}
\|h\|_{L^3_tL^6_x(\R \times \R^5_*)} \lesssim \|\vec h(0)\|_{\HH} \lesssim \de_0
}
Moreover, if we let $h_L(t):=S_0(t) \vec h(0)\in \HH$ denote the free linear evolution, i.e., solution to~\eqref{eq:free}, of the data $\vec h(0)$ we have 
\EQ{\label{lin nl}
\sup_{t \in \R}\|\vec h(t) - \vec h_L(t)\|_{\HH} \lesssim R^{-11/3}\|\vec h(0)\|_{\HH} + R^{-11/6}\|\vec h(0)\|_{\HH}^2 + \|\vec h(0)\|_{\HH}^3
}
\end{lem}

\begin{rem}
Note that for each $t \in \R$, 
\EQ{
\NN_R(t, r, h) = -V(r) h + F(r, h)+G(r, h)  \quad \forall r \ge R+ \abs{t}
} 
where $V(r)$, $F(r,h)$, and $G(r, h)$ are as in \eqref{eq:u nl}. By finite propagation speed it is then immediate that solutions to \eqref{eq:h} and \eqref{eq:u nl} agree on the exterior cone $\{(t, r) \mid r \ge R + \abs{t}\}$. 
\end{rem}

\begin{proof}[Proof of Lemma~\ref{m cp}]  The small data well-posedness theory, including estimate~\eqref{a priori}, follows from the usual contraction and continuity arguments based on the Strichartz estimates in Proposition~\ref{ext strich}.  To prove \eqref{lin nl} we note that by the Duhamel formula and 
Strichartz estimates we have 
\ant{
\|\vec h(t)- \vec h_L(t)\|_{\HH} &\lesssim \| \NN_R(\cdot, \cdot, h)\|_{L^1_tL^2_x(\R \times \R^5_*)}\\
& \lesssim \|V_R h\|_{L^1_tL^2_x(\R \times \R^5_*)} + \|F_R(\cdot, \cdot, h)\|_{L^1_tL^2_x(\R \times \R^5_*)} + \|G(\cdot, h)\|_{L^1_tL^2_x(\R \times \R^5_*)}
}
We can now estimate the three terms on the right-hand side above. First, we claim that  
\ant{ 
 \|V_R h\|_{L^1_tL^2_x(\R \times \R^5_*)} \lesssim \|V_R\|_{L^{\frac{3}{2}}_tL^3_x}\|h\|_{L^3_tL^6_x}  \lesssim R^{-11/3}\|h\|_{L^3_tL^6_x}  
 }
 To see this, we can use~\eqref{VR decay} to deduce that for each $t \in \R$
 \ant{
 \|V_R\|_{L^3_x}^{3} &\lesssim  \int_{1}^{R+\abs{t}}(R+ \abs{t})^{-18} r^4 \, dr + \int_{R+ \abs{t}}^{\I} r^{-18}r^4 \, dr \\
 & \lesssim (R+ \abs{t})^{-13}
 }
Therefore, 
\ant{
\|V_R\|_{L^{\frac{3}{2}}_tL^3_x} \lesssim \left(\int_{\R} (R+\abs{t})^{-13/2} \, dt \right)^{\frac{2}{3}} \lesssim R^{-11/3}
}
Similarly, we can show using \eqref{FR decay} and \eqref{G decay} that 
\ant{
& \|F_R(\cdot, \cdot, h)\|_{L^1_tL^2_x(\R \times \R^5_*)} \lesssim R^{-11/6}\|h\|_{L^3_tL^6_x}^2\\
 & \|G(\cdot, h)\|_{L^1_tL^2_x(\R \times \R^5_*)} \lesssim \|h\|^3_{L^3_tL^6_x}
}
 which proves \eqref{lin nl}.
\end{proof}
We can now prove Lemma~\ref{key ineq lem}. 

\begin{proof}[Proof of Lemma~\ref{key ineq lem}] We will first prove Lemma~\ref{key ineq} for time $t=0$.  The fact that \eqref{key ineq} holds at all times $t \in \R$ for $R> R_0$, with $R_0$ independent of $t$ will follow from  the pre-compactness of~$K$. 

For each $R \ge 1$, define  truncated initial data
 $\vec{u}_R(0) = (u_{0, R}, u_{1, R})$ given by
\ali{\label{R data}
&u_{0, R}(r) =  \begin{cases} u_0( r) \mfor r \ge R\\ \frac{u_0(R)}{R-1}(r-1) \mfor r < R,\end{cases}\\
& u _{1, R}(r) =  \begin{cases} u_1(r) \mfor r \ge R\\ 0 \mfor r <R. \end{cases}
}
Observe that this truncated data has small energy for large $R$ since
\ali{\label{same size}
\|\vec u_{R}(0)\|_{\HH} \lesssim \|\vec u(0) \|_{\HH(r \ge R)}. 
}
In particular, there exists $R_0\ge1$ so that for all $R \ge R_0$ we have 
\ant{
\|\vec u_{R}(0)\|_{\HH} \le \de_0 
}
where $\de_0$ is the small constant in Lemma~\ref{m cp}. Let $\vec u_R(t)$ denote the solution to \eqref{eq:h}  given by Lemma~\ref{m cp} with data $\vec u_R(0)$ as in \eqref{R data}. Note that by finite propagation speed we have  
\ant{
\vec u_R(t, r) = \vec u(t, r) \quad \forall t \in \R, \, \, \forall r \ge R+\abs{t}
} 
Also let $\vec u_{R, L}(t) = S_0(t) \vec u_R(0)$ denote the solution to free wave equation \eqref{eq:free} with initial data $\vec u_R(0)$. 
Now, by the triangle inequality we obtain for each $t \in \R$
\EQ{\label{tri ineq}
 \|\vec u(t)\|_{\HH(r \ge R + \abs{t})} = \|\vec u_R(t)\|_{\HH(r \ge R + \abs{t})}  &\ge \| \vec u_{R,L}(t) \|_{\HH(r \ge R + \abs{t})}\\
 &\quad - \| \vec u_R(t) - \vec u_{R, L}(t)\|_{\HH} 
}
By \eqref{lin nl} and \eqref{same size}  we can deduce that
\ant{
\sup_{t \in \R}\|\vec u_R(t) - \vec u_{R, L}(t)\|_{\HH} &\lesssim R^{-11/3}\|\vec u_R(0)\|_{\HH}+ R^{-11/6}\|\vec u_R(0)\|_{\HH}^2 + \|\vec u_R(0)\|_{\HH}^3\\
& \lesssim R^{-11/3}\|\vec u(0)\|_{\HH(r \ge R)} + R^{-11/6}\|\vec u(0)\|_{\HH(r \ge R)}^2 \\
& \quad+ \|\vec u(0)\|_{\HH(r \ge R)}^3
}
Therefore \eqref{tri ineq} gives
\ant{
 \|\vec u(t)\|_{\HH(r \ge R + \abs{t})} & \ge \| \vec u_{R,L}(t) \|_{\HH(r \ge R + \abs{t})}    -C_0R^{-11/3}\|\vec u(0)\|_{\HH(r \ge R)}\\
 & \quad - C_0R^{-11/6}\|\vec u(0)\|_{\HH(r \ge R)}^2 - C_0\|\vec u(0)\|_{\HH(r \ge R)}^3
}
Letting $t $ tend to either $\pm \infty$ -- the choice determined by Proposition~\ref{linear prop} -- we can use  Proposition~\ref{linear prop} to estimate the right-hand side above and use Corollary~\ref{compact} to see that the left-hand side above tends to zero, which  gives
\ant{
\| \pi^{\perp}_{R}\, \vec u_R(0)\|_{\HH(r \ge R)}^2 &\lesssim R^{-22/3}\|\vec u(0)\|_{\HH(r \ge R)}^2 + R^{-11/3}\|\vec u(0)\|_{\HH(r \ge R)}^4 + \|\vec u(0)\|_{\HH(r \ge R)}^6 
}
after squaring both sides. Finally we note that by the definition of $\vec u_R(0)$, $$\| \pi^{\perp}_{R}\, \vec u_R(0)\|_{\HH(r \ge R)}^2 = \| \pi^{\perp}_{R}\, \vec u(0)\|_{\HH(r \ge R)}^2$$ Therefore, 
\ant{
\| \pi^{\perp}_{R}\, \vec u(0)\|_{\HH(r \ge R)}^2& \lesssim R^{-22/3}\left(\| \pi_{R}\,\vec u(0)\|_{\HH(r \ge R)}^2 +\| \pi^{\perp}_{R}\,\vec u(0)\|_{\HH(r \ge R)}^2\right)\\
&\quad+ R^{-11/3}\left(\| \pi_{R}\,\vec u(0)\|_{\HH(r \ge R)}^2 +\| \pi^{\perp}_{R}\,\vec u(0)\|_{\HH(r \ge R)}^2\right)^2 \\
& \quad+ \left(\| \pi_{R}\,\vec u(0)\|_{\HH(r \ge R)}^2 +\| \pi^{\perp}_{R}\,\vec u(0)\|_{\HH(r \ge R)}^2\right)^3
}
where we have used the orthogonality of the projection $\pi_{R}\,$ to expand the right-hand side. To conclude the proof, simply choose $R_0$ large enough so that we can absorb all of the terms on the right-hand side involving $\pi^{\perp}$ into the left-hand side and deduce that  
\ant{
\| \pi^{\perp}_{R}\, \vec u(0)\|_{\HH(r \ge R)}^2 &\lesssim R^{-22/3}\| \pi_{R}\, \vec u(0)\|_{\HH(r \ge R)}^2 \\
&\quad + R^{-11/3}\| \pi_{R}\, \vec u(0)\|_{\HH(r \ge R)}^4 +\| \pi_{R}\, \vec u(0)\|_{\HH(r \ge R)}^6.
} 
This proves Lemma~\ref{key ineq lem} for $t=0$. To show that this inequality holds for all $t \in \R$ observe that by the pre-compactness of $K$ we can choose $R_0 =R_0( \de_0)$ so that 
\EQ{
\| \vec u(t)\|_{ \HH ( r \ge R)} \le \de_0
} uniformly in $t \in \R$. Now simply repeat the argument given above with the truncated initial data for time $t=t_0$ and $R \ge R_0$ defined by 
\ant{
&u_{0, R, t_0}(r) =  \begin{cases} u(t_0, r) \mfor r \ge R\\ \frac{u(t_0,R)}{R-1}(r-1) \mfor r < R,\end{cases}\\
& u _{1, R, t_0}(r) =  \begin{cases} u_t(t_0, r) \mfor r \ge R\\ 0 \mfor r <R. \end{cases}
}
This concludes the argument.
\end{proof}

\subsection{Step 2} In this step we will deduce the asymptotic behavior of $\vec{u}(0,r)$ as $r \to \infty$ described in \eqref{asymp}. In particular we will establish the following result. 
\begin{lem}\label{asymp lem} Let $\vec u(t)$ be as in Proposition~\ref{rigidity} with $\vec u(0) = (u_0, u_1)$. Then there exists $\ell_0  \in \R$ such that  
\begin{align}
&r^3 u_0(r) \to  \ell_0 \mas r \to \infty \\
&r \int_r^{\infty} u_1( \rho) \rho \, d\rho \to 0 \mas r \to \infty
\end{align} 
Moreover, we have the following estimates for the rates of convergence,
\begin{align} 
&\abs{r^3 u_0(r) - \ell_0 } = O(r^{-3}) \mas r \to \infty\\
&\abs{r \int_r^{\infty} u_1( \rho) \rho \, d\rho} = O(r^{-1}) \mas r \to \infty 
\end{align}
\end{lem}

To begin, we define 
\EQ{\label{v def}
&v_0(t, r) := r^3 u(t,r)\\
&v_1(t, r):= r \int_r^{\infty} u_t(t, \rho) \rho \, d\rho
}
and for simplicity we will write $v_0(r):=v_0(0, r)$ and $v_1(r):=v_1(0,r)$. By direct computation  one can show that 
\begin{align}\label{pi perp}
&\| \pi^{\perp}_{R}\, \vec u(t)\|_{\HH(r \ge R)}^2 = \int_R^{\infty} \left(\frac{1}{r} \p_r v_0(t, r)\right)^2 \, dr + \int_R^{\I} (\p_r v_1(t, r))^2 \, dr\\
&\| \pi_{R}\, \vec u(t)\|_{\HH(r \ge R)}^2 = 3R^{-3}v_0^2(t,R) + R^{-1}v_1^2(t, R) \label{pi}
\end{align}
For convenience, we can rewrite the conclusions of Lemma~\ref{key ineq lem} in terms of $(v_0,v_1)$: 
\begin{lem} \label{v ineq}Let $(v_0, v_1)$ be defined as in \eqref{v def}. There exists $R_0 >1$ so that for all $R>R_0$ we have 
\begin{align*}
\int_R^{\infty} \left(\frac{1}{r} \p_r v_0(t, r)\right)^2 \, dr +& \int_R^{\I} (\p_r v_1(t, r))^2 \, dr   \lesssim R^{-\frac{22}{3}}\left(3R^{-3}v_0^2(t,R) + R^{-1}v_1^2(t, R)\right) \\
&+  R^{-\frac{11}{3}}\left(3R^{-3}v_0^2(t,R) + R^{-1}v_1^2(t, R)\right)^2 \\
&+ \left(3R^{-3}v_0^2(t,R) + R^{-1}v_1^2(t, R)\right)^3 \\
&\lesssim R^{-\frac{31}{3}}v_0^2(t, R) + R^{-\frac{29}{3}} v_0^4(t, R) + R^{-9}v_0^6(t, R)\\
&+R^{-\frac{25}{3}}v_1^2(t, R) + R^{-\frac{17}{3}} v_1^4(t,R) + R^{-3}v_1^6(t, R)
\end{align*}
with the above estimates holding uniformly in $t \in \R$. 
\end{lem}

We will use Lemma~\ref{v ineq} to prove a difference estimate. First, let $\de_1>0$ be a small number to be determined below with $\de_1 \le \de_0$ where $\de_0$ is as in Lemma~\ref{m cp}. Let $R_1$ be large enough so that for all $R \ge R_1$ we have 
\EQ{\label{R1 def}
&\|\vec u(t) \|_{\HH(r \ge R)} \le \de_1 \le \de_0  \quad \forall R \ge R_1, \quad \forall t \in \R\\
&R_1^{-\frac{11}{3}} \le \de_1
}
We note again that such an $R_1=R_1( \de_1)$ exists by the pre-compactness of $K$.

\begin{cor}\label{diff cor1} Let $R_1$ be as above. The for all $r, r'$ with $R_1 \le r \le r' \le 2r$ and for all $t \in \R$ we have
\EQ{ \label{v0 diff0}
\abs{v_0(t, r) - v_0(t, r')} &\lesssim r^{-\frac{11}{3}} \abs{v_0(t,r)} + r^{-\frac{10}{3}}\abs{v_0(t, r)}^2 + r^{-3} \abs{v_0(t, r)}^3\\
& \, + r^{-\frac{8}{3}}\abs{v_1(t, r)} + r^{-\frac{4}{3}} \abs{v_1(t,r)}^2+ \abs{v_1(t, r)}^3
}
and 
\EQ{ \label{v1 diff0}
\abs{v_1(t, r) - v_1(t, r')} &\lesssim r^{-\frac{14}{3}} \abs{v_0(t,r)} + r^{-\frac{13}{3}}\abs{v_0(t, r)}^2 + r^{-4} \abs{v_0(t, r)}^3\\
& \, + r^{-\frac{11}{3}}\abs{v_1(t, r)} + r^{-\frac{7}{3}} \abs{v_1(t,r)}^2 +r^{-1} \abs{v_1(t, r)}^3
}
with the above estimates holding uniformly in $t \in \R$. 
\end{cor}

We will also need a trivial consequence of the preceding result which we state as another corollary for convenience. 
\begin{cor}\label{diff cor2} Let $R_1$ be as above. The for all $r, r'$ with $R_1 \le r \le r' \le 2r$ and for all $t \in \R$ we have
\EQ{
\abs{v_0(t, r) - v_0(t, r')} &\lesssim  \de_1 \abs{v_0(t, r)} +  r \de_1\abs{v_1(t, r)} \label{v0 diff}
}
and 
\EQ{
\abs{v_1(t, r) - v_1(t, r')} &\lesssim r^{-1} \de_1\abs{ v_0(t,r)} +  \de_1\abs{v_1(t, r)} \label{v1 diff}
}
with the above estimates holding uniformly in $t \in \R$. 
\end{cor}
We remark that Corollary~\ref{diff cor2} follows immediately from Corollary~\ref{diff cor1} in light of   \eqref{pi}  and \eqref{R1 def}. 

\begin{proof}[Proof of Corollary~\ref{diff cor1}]This is a simple consequence of Lemma~\ref{v ineq}. Indeed, for $r \ge R_1$ and $r' \in [r, 2r]$ we use Lemma~\ref{v ineq} to see that
\ant{
\abs{v_0(t, r) -v_0(t, r')}^2 &\le \left(\int_r^{r'} \abs{\p_r v_0(t, \rho)} \, d \rho\right)^2 \\
& \le \left( \int_r^{r'} \rho^2 \, d\rho \right)\left(\int_r^{r'} \abs{\frac{1}{ \rho}\p_r v_0(t, \rho)}^2  \, d \rho\right) \\
& \lesssim r^3 \left(r^{-\frac{31}{3}}v_0^2(t, r) + r^{-\frac{29}{3}} v_0^4(t, r) + r^{-9}v_0^6(t, r)\right)\\
&+r^3 \left(r^{-\frac{25}{3}}v_1^2(t, r) + r^{-\frac{17}{3}} v_1^4(t,r) + r^{-3}v_1^6(t, r)\right)
}
Similarly, 
\ant{
\abs{v_1(t, r) -v_1(t, r')}^2 &\le \left(\int_r^{r'} \abs{\p_r v_1(t, \rho)} \, d \rho\right)^2 \\
& \le \left( \int_r^{r'}  \, d\rho \right)\left(\int_r^{r'} \abs{\p_r v_1(t, \rho)}^2  \, d \rho\right) \\
& \lesssim r \left(r^{-\frac{31}{3}}v_0^2(t, r) + r^{-\frac{29}{3}} v_0^4(t, r) + r^{-9}v_0^6(t, r)\right)\\
&+r \left(r^{-\frac{25}{3}}v_1^2(t, r) + r^{-\frac{17}{3}} v_1^4(t,r) + r^{-3}v_1^6(t, r)\right)
}
as claimed.
\end{proof} 

The next step towards establishing Lemma~\ref{asymp lem} is to provide an upper bound on the growth rates of $v_0(t,r)$ and $v_1(t,r)$. 

\begin{claim}\label{slow growth}
Let $v_0(t, r)$, $v_1(t, r)$ be as in \eqref{v def}. Then,
\begin{align}\label{v0 grow}
&\abs{v_0(t, r)} \lesssim r^{\frac{1}{6}}\\
&\abs{v_1(t, r)} \lesssim r^{\frac{1}{18}}\label{v1 grow}
\end{align}
uniformly in $t\in \R$. 
\end{claim}

\begin{proof}
First, note that it suffices to prove  Claim~\ref{slow growth} only for $t=0$ since the ensuing argument  relies exclusively on results in this section that  hold uniformly in $t \in \R$. Fix $r_0 \ge R_1$ and observe that by \eqref{v0 diff}, \eqref{v1 diff} 
\begin{align}\label{v0 1}
\abs{v_0( 2^{n+1}r_0)} &\le  (1+C_1\de_1) \abs{v_0(2^nr_0)} +  (2^nr_0) C_1\de_1\abs{v_1( 2^nr_0)}\\
\abs{v_1( 2^{n+1}r_0)} &\le  (1+C_1\de_1) \abs{v_1(2^nr_0)} +  (2^nr_0)^{-1} C_1\de_1\abs{v_0( 2^nr_0)} \label{v1 1}
\end{align}
To simply the exposition, we introduce the notation 
\begin{align} \label{an def}
&a_n:= \abs{v_1( 2^{n}r_0)}\\
&b_n:= (2^{n}r_0)^{-1}\abs{v_0( 2^{n}r_0)} \label{bn def}
\end{align}
Then, combining \eqref{v0 1} and \eqref{v1 1} gives 
\ant{
a_{n+1} + b_{n+1} &\le \left(1+ \frac{3}{2} C_1 \de_1\right) a_n + \left(\frac{1}{2} + \frac{3}{2}C_1 \de_1\right) b_n\\
&\le  \left(1+ \frac{3}{2} C_1 \de_1\right)(a_n + b_n)
}
Arguing inductively we then see that for each $n$ we have
\ant{
(a_{n} + b_{n} )\le \left(1+ \frac{3}{2} C_1 \de_1\right)^n(a_0 + b_0)
}
Choosing $\de_1$ small enough so that $(1+ \frac{3}{2}C_1\de_1) \le 2^{\frac{1}{18}}$ allows us to conclude that
\EQ{ \label{an grow}
a_n \le C(2^n r_0)^{\frac{1}{18}}
}
where the constant $C>0$ above depends on $r_0$ which is fixed. In light of \eqref{an def} we have thus proved~\eqref{v1 grow} for all $r= 2^nr_0$. Now define
\EQ{ \label{cn def}
c_n:= \abs{v_0(2^n r_0)}
}
By \eqref{pi},  \eqref{R1 def}, \eqref{v0 diff0}, and \eqref{an grow} we have 
\ant{
c_{n+1} \le (1+ C_1 \de_1)c_n + C(2^nr_0)^{\frac{1}{6}}
}
Inductively, we can deduce that 
\ant{
c_n &\le (1+C_1\de_1)^n c_0 + Cr_0^{\frac{1}{6}} \sum_{k=1}^{n} (1+ C_1\de_1)^{n-k}2^{\frac{k-1}{6}}\\
& \le C (2^nr_0)^{\frac{1}{6}}
}
where we have used that $(1+ C_1\de_1) \le 2^{\frac{1}{18}}$, and again the constant $C>0$ depends on $r_0$, which is fixed. This proves~\eqref{v0 grow} for $r=2^nr_0$. The general estimates \eqref{v0 grow} and \eqref{v1 grow} follow from the difference estimates \eqref{v0 diff0}  and \eqref{v1 diff0}. 
\end{proof}

\begin{claim}\label{v1 ell1}
For each $t \in \R$ there exists a number $\ell_1(t) \in \R$ such that 
\ali{ \label{v1 lim}
 \abs{v_1(t, r)- \ell_1(t)}= O(r^{-1}) \mas r  \to \infty
}
where the $O(\cdot)$ is uniform in~$t$. 
\end{claim}
\begin{proof}
Again, it suffices to show this for $t=0$. Let $r_0 \ge R_1$ where $R_1>1$ is as in \eqref{R1 def}. By \eqref{v1 diff0} and Claim~\ref{slow growth} we have 
\EQ{\nn
\abs{v_1(2^{n+1}r_0)- v_1(2^n r_0)} &\lesssim (2^n r_0)^{-\frac{9}{2}} + (2^nr_0)^{-4} + (2^nr_0)^{-\frac{7}{2}} \\
& \, + (2^nr_0)^{-\frac{65}{18}} + (2^nr_0)^{-\frac{20}{9}}  +(2^nr_0)^{-\frac{5}{6}}\\
& \lesssim (2^{n}r_0)^{-\frac{5}{6}}
}
This implies that the series 
\ant{
\sum_n \abs{v_1(2^{n+1}r_0)- v_1(2^n r_0)}< \infty
} 
which in turn implies that there exists $\ell_1 \in \R$ such that $$\lim_{n \to \infty} v_1(2^nr_0) =  \ell_1.$$ The fact that $\ds{\lim_{r \to \infty} v_1(r) = \ell_1}$ follows from the difference estimates \eqref{v0 diff0}, \eqref{v1 diff0}, and the growth estimates \eqref{v0 grow}, \eqref{v1 grow}. To establish the estimates on the rate of convergence in \eqref{v1 lim} we note that by the difference estimate \eqref{v1 diff0} and the fact that we now know that $\abs{v_1(r)}$ is bounded, for large enough $r$ we have 
\ant{
\abs{v_1(2^{n+1}r)- v_1(2^n r)} \lesssim (2^n r)^{-1}
} 
Hence, 
\ant{
\abs{v_1(r)- \ell_1} = \abs{ \sum_{n \ge 0} (v_1(2^{n+1}r)- v_1(2^n r))} \lesssim r^{-1}  \sum_{n \ge 0}2^{-n} \lesssim r^{-1}
}
as desired.
\end{proof}

Next we show that the limit $\ell_1(t)$ is actually independent of $t$. 
\begin{claim}
The function $\ell_1(t)$ in Claim~\ref{v1 ell1} is independent of $t$, i.e., $\ell_1(t) = \ell_1$ for all $t \in \R$. 
\end{claim}
\begin{proof}
By the definition of $v_1(t, r)$ we have shown that 
\ant{
\ell_1(t) = r \int_r^{\infty} u_t(t, \rho) \rho \, d \rho + O(r^{-1})
}
Fix $t_1, t_2 \in \R$ with $t_1 \neq t_2$. We will show that 
\ant{
\ell_1(t_2) - \ell_1(t_1) = 0
}
To see this observe that for each $R \ge R_1$ we have
\EQ{ \nn
\ell_1(t_2) - \ell_1(t_1) &= \frac{1}{R} \int_R^{2R} (\ell_1(t_2) - \ell_1(t_1)) \, ds\\
&= \frac{1}{R} \int_R^{2R}\left(s \int_s^{\infty} (u_t(t_2, r) - u_t(t_1, r)) r \, dr \right) \, ds + O(R^{-1})\\
&=\frac{1}{R} \int_R^{2R}\left(s \int_s^{\infty}  \int_{t_1}^{t_2}u_{tt}(t, r) \, dt\,  r \, dr \right) \, ds + O(R^{-1})
}
Using the fact that $u$ is a solution to \eqref{eq:u nl}, we can rewrite the above integral as
\EQ{
&=  \int_{t_1}^{t_2}\frac{1}{R} \int_R^{2R}\left(s \int_s^{\infty}  (ru_{rr}(t, r) + 4u_r(t, r) ) \, dr \right) \, ds\, dt + \\
&\, +  \int_{t_1}^{t_2} \frac{1}{R} \int_R^{2R}\left(s \int_s^{\infty}(-rV(r)u(t, r)+ rN(r, u(t, r)))  \, dr \right) \,ds\, dt \\
&\, + O(R^{-1})\\
& = \Rmnum{1} + \Rmnum{2} +  O(R^{-1})
}
To estimate $I$ we integrate by parts: 
\EQ{ \label{I}
I&= \int_{t_1}^{t_2}\frac{1}{R} \int_R^{2R}\left(s \int_s^{\infty}  \frac{1}{r^3} \p_r(r^4 u_r(t, r)) \, dr \right) \, ds\, dt\\
& = 3\int_{t_1}^{t_2}\frac{1}{R} \int_R^{2R}\left(s \int_s^{\infty}  u_r(t, r) \, dr \right) \, ds\, dt -\int_{t_1}^{t_2}\frac{1}{R} \int_R^{2R}s^2  u_r(t, s) \, ds\, dt \\
& = -3\int_{t_1}^{t_2}\frac{1}{R} \int_R^{2R}r \, u(t, r) \, dr\, dt -\int_{t_1}^{t_2}\frac{1}{R} \int_R^{2R}r^2\, u_r(t, r) \, dr\, dt \\
&= -\int_{t_1}^{t_2}\frac{1}{R} \int_R^{2R}r \, u(t, r) \, dr\, dt + \int_{t_1}^{t_2} (Ru(t,R)-2Ru(t, 2R))\, dt
}
Finally, we note that \eqref{v0 grow} and the definition of $v_0(t, r)$ give us
\EQ{ \label{u decay}
r^3 \abs{u(t, r)} = \abs{v_0(t, r)} \lesssim r^{\frac{1}{6}}
}
Using this estimate for $\abs{u(t, r)}$ in the last line in \eqref{I} shows that 
\ant{
I=\abs{t_2-t_1}O(R^{-\frac{11}{6}})
}
To estimate $II$ we can use \eqref{u decay} to see that for $r>R$ large enough 
\ant{
\abs{-V(r)u(t, r) + N(r, u(t, r))} &\lesssim r^{-6}\abs{u(t, r)}+  r^{-3} \abs{u(t, r)}^2 +  \abs{u(t, r)}^3\\
& \lesssim r^{-6 - \frac{17}{6}} + r^{-3 -\frac{17}{3}} + r^{-\frac{17}{2}}\\
&\lesssim r^{-8}
}
Hence, 
\ant{
II &\lesssim  \int_{t_1}^{t_2} \frac{1}{R} \int_R^{2R}s \int_s^{\infty}r^{-8}  \, dr  \,ds\, dt
= \abs{t_2-t_1}O(R^{-6})
}
Putting this together we get 
\ant{
\abs{\ell_1(t_2) - \ell_1(t_1)} = O(R^{-1})
} 
which implies that $\ell_1(t_2) = \ell_1(t_1)$. 
\end{proof}

We next show that $\ell_1$ is necessarily equal to $0$. 
\begin{claim}  $ \ell_1 =0$. 
\end{claim} 
\begin{proof} Suppose $\ell_1 \neq 0$. 
We know that for all $R \ge R_1$  and for all $t \in \R$ we have 
\ant{
R\int_R^{\infty}u_t(t, r) \, r \, dr = \ell_1 + O(R^{-1})
}
where $O(\cdot)$ is uniform in~$t$. Hence, for $R$ large, the left-hand side above has the same sign as $\ell_1$, for all $t$.
Thus we can choose $R \ge R_1$ large enough so that for all $t \in \R$, 
\ant{
\Big| R\int_{R}^{\infty}u_t(t, r) \, r \, dr \Big|\ge \frac{ |\ell_1|}{2} 
}
Integrating from $t=0$ to $t=T$ gives
\ant{
\Big| \int_0^TR\int_{R}^{\infty}u_t(t, r) \, r \, dr\, dt\Big|  \ge T\frac{|\ell_1|}{2} 
}
However,  we integrate in $t$ on the left-hand side and use \eqref{u decay} to obtain
\ant{
\abs{R\int_{R}^{\infty}\int_0^Tu_t(t, r) \, r \, dt \, dr }&= \abs{ R\int_R^{\infty}[u(T, r) - u(0,r) ]\, r \, dr}\\
& \lesssim R\int_R^{\infty} r^{-\frac{11}{6}}\, dr
\lesssim R^{\frac{1}{6}}
}
Therefore for fixed large $R$  we have 
\ant{
T\frac{|\ell_1|}{2}  \lesssim R^{\frac{1}{6}}
}
which gives a contradiction by taking $T$ large. 
\end{proof}
Now that we have shown that $v_1(r) \to 0 $ as $r \to \infty$, we can prove that $v_0(r)$ also converges and complete the proof of Lemma~\ref{asymp lem}. 

\begin{proof}[Proof of Lemma~\ref{asymp lem}]
It remains to show that there exists $\ell_0 \in \R$ such that 
\ali{ 
\abs{v_0(r)- \ell_0} = O(r^{-3}) \mas r \to \infty
}
Using the difference estimate \eqref{v0 diff0} as well as \eqref{v0 grow} and the  fact that $\abs{v_1(r)} \lesssim r^{-1}$ for $r \ge R_1$ we have for $r_0 \ge R_1$
\ant{
\abs{v_0( 2^{n+1}r_0) - v_0( 2^{n}r_0)} &\lesssim (2^nr_0)^{-\frac{11}{3}} (2^nr_0)^{\frac{1}{6}}+ (2^nr_0)^{-\frac{10}{3}}(2^nr_0)^{\frac{1}{3}} + (2^nr_0)^{-3} (2^nr_0)^{\frac{1}{2}}\\
& \, + (2^nr_0)^{-\frac{8}{3}}(2^nr_0)^{-1}+ (2^nr_0)^{-\frac{4}{3}} (2^nr_0)^{-2} + (2^nr_0)^{-3}\\
&\lesssim (2^nr_0)^{-\frac{5}{2}}
}
Hence, 
\ant{ 
\sum_{n \ge 0} \abs{v_0( 2^{n+1}r_0) - v_0( 2^{n}r_0)} < \infty
}
and therefore there exists $\ell_0 \in \R$ so that 
\ant{
 \lim_{n \to \infty} v_0(2^nr_0) = \ell_0
 }
 By the difference estimate \eqref{v0 diff0} and the fact that $v_1(r) \to 0$ we can conclude that in fact $\ds{\lim_{r \to \infty}v_0(r) = \ell_0}$. To establish the convergence rate, we note that since we now know that $\abs{v_0(r)}$ is bounded we have the improved difference estimate
  \EQ{
 \abs{v_0( 2^{n+1}r) - v_0( 2^{n}r)} \lesssim (2^nr)^{-3}
 }
 which holds for all $r \ge R$. 
 Therefore, 
 \EQ{
 \abs{v_0(r)- \ell_0} = \abs{ \sum_{n \ge0} (v_0(2^{n+1}r)-v_0(2^nr))} \lesssim r^{-3} \sum_{n \ge 0} 2^{-3n}
 }
 as claimed. 
 \end{proof}

\subsection{Step 3} Finally, we complete the proof of Proposition~\ref{rigidity} by showing that $\vec u(t) = (0, 0)$. We divide this argument into two separate cases depending on whether the number $\ell_0$ found in the previous step is zero or nonzero. 

\begin{flushleft} \textbf{Case 1: $\ell_0=0$ implies $\vec u(0)=(0, 0)$:} 
\end{flushleft}
In this case we show that if $\ell_0 = 0$, then $\vec u(t) = (0, 0)$. 
\begin{lem}\label{ell=0 lem}
Let $\vec u(t)$ be as in Proposition~\ref{rigidity} and let $\ell_0$ be as in Lemma~\ref{asymp lem}. Suppose that $\ell_0=0$.  Then $\vec u(t) =(0, 0)$. 
\end{lem}

We begin by showing that if $\ell_0=0$ then $(u_0, u_1)$ must be compactly supported. 
\begin{claim}\label{comp supp}
Let $\ell_0$ be as in Lemma~\ref{asymp lem}. If $\ell_0=0$ then  $(u_0, u_1)$ must be compactly supported.
\end{claim}
\begin{proof}
The assumption $\ell_0=0$ means that 
\EQ{ \label{assume}
&\abs{v_0(r)}= O(r^{-3}) \mas r \to \infty\\
& \abs{v_1(r)} = O(r^{-1}) \mas r \to \infty
}
Therefore, for $r_0 \ge R_1$ we have 
\EQ{\label{v less}
\abs{v_0(2^n r_0)} + \abs{ v_1(2^nr_0)} \lesssim (2^nr_0)^{-3} + (2^nr_0)^{-1} \lesssim (2^nr_0)^{-1}
}
On the other hand, using the difference estimates \eqref{v0 diff0}--\eqref{v1 diff} as well as our assumption \eqref{assume} we obtain
\ant{
&\abs{v_0(2^{n+1}r_0) }\ge \left(1- C_1\de_1\right)\abs{v_0(2^nr_0) }- C_1(2^nr_0)^{-2} \abs{v_1(2^nr_0)}\\
&\abs{v_1(2^{n+1}r_0)} \ge \left(1- C_1\de_1\right)\abs{v_1(2^nr_0)} - C_1(2^nr_0)^{-4} \abs{v_0(2^nr_0)}
}
This means that 
\ant{
\abs{v_0(2^{n+1}r_0)} + \abs{v_1(2^{n+1}r_0)} \ge (1- C_1 \de_1 -C_1r_0^{-2})  \left(\abs{v_0(2^{n}r_0)} + \abs{v_1(2^{n}r_0)}\right)
}
Choose $r_0$ large enough and $\de_1$ small enough so that $C_1 (\de_1 + r_0^{-2}) < \frac{1}{4}$. Arguing inductively we can conclude that 
\ant{
\abs{v_0(2^{n}r_0)} + \abs{v_1(2^{n}r_0)} \ge \left(\frac{3}{4}\right)^n\left(\abs{v_0(r_0)} + \abs{v_1(r_0)}\right)
}
Estimating the left hand side above using \eqref{v less} gives 
\ant{
\left(\frac{3}{4}\right)^n\left(\abs{v_0(r_0)} + \abs{v_1(r_0)}\right) \lesssim 2^{-n}r_0^{-1}
}
which means that 
\ant{
\left(\frac{3}{2}\right)^n \left(\abs{v_0(r_0)} + \abs{v_1(r_0)} \right) \lesssim 1
}
Hence $\vec v(r_0):=(v_0(r_0), v_1(r_0)) =(0, 0)$. But then \eqref{pi}  implies that 
\ant{
\| \pi_{r_0} \vec u(0)\|_{\HH(r \ge r_0)} = 0
}
Using Lemma~\ref{key ineq lem} we can also deduce that 
\ant{
\| \pi^{\perp}_{r_0} \vec u(0)\|_{\HH(r \ge r_0)} = 0
}
and hence 
\ant{
\| \vec u(0)\|_{\HH(r \ge r_0)} = 0
}
which concludes the proof since $\ds\lim_{r \to \infty}u_0(r)=0$. 
\end{proof}


\begin{proof}[Proof of Lemma~\ref{ell=0 lem}] Assume that $\ell_0=0$. Then by Claim~\ref{comp supp},  $(u_0, u_1)$ is compactly supported.  We assume that $(u_0, u_1) \neq (0, 0)$ and argue by contradiction. In this case we can find
 $\rho_0>1$ so that 
 \ant{
 \rho_0 := \inf \{ \rho \, : \,  \|\vec u(0)\|_{\HH(r \ge \rho)} =  0\}
 }
Let $ \e>0$ small to be determined below and find $1<\rho_1 < \rho_0$, $\rho_1= \rho_1( \e)$ so  that 
\ant{
0< \| \vec u(0) \|_{\HH(r \ge \rho_1)}^2 \le \e \le \de_1^2
} 
where $\de_1>0$ is as in \eqref{R1 def}. With $(v_0, v_1)$ as in \eqref{v def} we have 
\begin{multline} \label{pi pi perp}
\int_{\rho_1}^{\infty} \left(\frac{1}{r} \p_r v_0(r)\right)^2 \, dr + \int_{\rho_1}^{\I} (\p_r v_1( r))^2 \, dr+ 3\rho_1^{-3}v_0^2(\rho_1) + \rho_1^{-1}v_1^2( \rho_1)  =\\
= \| \pi_{\rho_1}^{\perp}\vec u(0)\|_{\HH(r \ge \rho_1)}^2 +\| \pi_{\rho_1} \vec u(0)\|_{\HH(r \ge \rho_1)}^2 = \|  \vec u(0)\|_{\HH(r \ge \rho_1)}^2 < \e
\end{multline}
By Lemma~\ref{v ineq} we also have 
\begin{multline} \label{key ineq3}
\int_{\rho_1}^{\infty} \left(\frac{1}{r} \p_r v_0(r)\right)^2 \, dr + \int_{\rho_1}^{\I} (\p_r v_1( r))^2 \, dr \lesssim \rho_1^{-\frac{31}{3}}v_0^2(\rho_1) + \rho_1^{-\frac{29}{3}} v_0^4(\rho_1) + \rho_1^{-9}v_0^6(\rho_1)\\
+\rho_1^{-\frac{25}{3}}v_1^2(\rho_1) + \rho_1^{-\frac{17}{3}} v_1^4(\rho_1) + \rho_1^{-3}v_1^6(\rho_1)
\end{multline}
Arguing as in Corollary~\ref{diff cor2} and using the fact that $v_0( \rho_0) = v_1( \rho_0) = 0$ gives
\EQ{
\abs{v_0( \rho_1)} = \abs{v_0(\rho_1) - v_0(\rho_0)} &\lesssim  \e \abs{v_0(\rho_1)} +  \rho_1 \e\abs{v_1(\rho_1)} \label{v0 diff2}
}
and 
\EQ{
\abs{v_1( \rho_1)}=\abs{v_1(\rho_1) - v_1(\rho_0)} &\lesssim \rho_1^{-1} \e \abs{v_0(\rho_1)} +  \e\abs{v_1(\rho_1)} \label{v1 diff2}
}
Plugging \eqref{v0 diff2} into \eqref{v1 diff2} gives 
\ant{
\abs{v_1( \rho_1)} \lesssim \rho_1^{-1} \e^2 \abs{v_0(\rho_1)} + \e(1+\e) \abs{v_1( \rho_1)}
} 
which means that for $\e$ small enough we have 
\ali{\label{v1 v0}
\abs{v_1( \rho_1)} \lesssim \rho_1^{-1} \e^2 \abs{v_0(\rho_1)}
}
Putting this estimate back into \eqref{v0 diff2} we obtain
\ant{
\abs{v_0( \rho_1)} \lesssim \e \abs{v_0(\rho_1)} +  \e^3 \abs{v_0(\rho_1)} \lesssim \e(1+ \e^2) \abs{v_0 (\rho_1)}
}
which implies that $v_0( \rho_1) = 0$ as long as $\e$ is chosen small enough. By \eqref{v1 v0} we can conclude that $v_1( \rho_1)=0$ as well. By \eqref{key ineq3} and \eqref{pi pi perp} we then have that 
\ant{
\|\vec u(0)\|_{\HH(r \ge \rho_1)} = 0
} 
which is a contradiction since $\rho_1 < \rho_0$. 
\end{proof}

We next consider the case $\ell_0 \neq 0$. 
\begin{flushleft}\textbf{Case $2$: $\ell_0 \neq 0$ is impossible.}
\end{flushleft}
In this final step we show that the case $\ell_0 \neq 0$ is impossible. Indeed we prove that if $\ell_0 \neq 0$ then our original wave map $\vec \psi(t)$ is equal to a rescaled solution~$Q_{\ell_0}$ to~\eqref{hm} that does not satisfy the Dirichlet boundary condition, $Q_{\ell_0}(1) \neq 0$, which is a contradiction since $\psi(t, 1)= 0$ for all $t \in \R$. 


We have shown that 
\ant{
r^3u_0(r) =  \ell_0 + O(r^{-3})
}
Recall that $ru_0(r) = \fy_0(r) = \psi_0(r) - Q(r)$ and that  $$Q(r) = n\pi - \frac{\al_0}{r^2} + O(r^{-6})$$ where $\al_0>0$ is \emph{uniquely} determined by the boundary condition $Q(1) = 0$. Hence, 
\ali{
\psi_0(r) = n\pi  - \frac{\al_0- \ell_0}{r^2} + O(r^{-5})
}
By Lemma~\ref{ode lem} there is a solution $Q_{\al_0- \ell} \in \dot{H}^1(\R_*^3)$ to \eqref{hm} satisfying 
\EQ{ \label{Q ell def}
Q_{\al_0- \ell_0}(r) =  n \pi - \frac{\al_0 - \ell_0}{r^2} + O(r^{-6})
}
and from here out we write $Q_{\ell_0} := Q_{\al_0- \ell_0}$. Note, by Lemma~\ref{ode lem},  $\ell_0\neq 0$ implies that $$Q_{\ell_0}(1) \neq 0$$ Indeed, recall from the discussion following Lemma~\ref{ode lem} that if $\al_0- \ell_0>0$ then $Q_{\ell_0}$ is a nontrivial rescaling of the harmonic map $Q$ and hence no longer satisfies the boundary condition. If $\al_0- \ell_0 =0$ then $Q_{\ell_0}(r) = n\pi$ for all $r$. Finally, we recall  that $\al_0 - \ell_0<0$ implies that  $Q_{\ell_0}(r)>n \pi$ for all $r$. 
Now set
\ali{ \label{uell def}
 &u_{\ell_0, 0}(r) :=   \frac{1}{r}( \psi_0( r) - Q_{\ell_0}(r))\\
&u_{\ell_0, 1}(r):= \frac{1}{r} \psi_1(r)
 }
 For each $t \in \R$ define $u_{\ell_0}(t, r) := \frac{1}{r}( \psi(t, r) - Q_{\ell_0}(r))$.  We record a few properties of $ \vec  u_{\ell_0}:= (u_{\ell_0}, \p_t u_{\ell_0})$. Note that by construction we have 
 \ali{ \label{v ell def}
& v_{\ell_0, 0}(r):= r^3u_{\ell_0}(r) = O(r^{-3}) \mas r \to \infty\\
&v_{\ell_0, 1}(r) := r \int_r^{\infty} \rho\,  u_{\ell_0, 1}( \rho) \, d \rho = O(r^{-1}) \mas r \to \infty
 }
 Also, $\vec u_{\ell_0}(t)$ satisfies 
 \EQ{\label{eq: u ell}
 &\p_{tt}u_{\ell_0}- \p_{rr} u_{\ell_0} - \frac{4}{r}\p_ru_{\ell_0} =  -V_{\ell_0}(r)u +N_{\ell_0}(r, u_{\ell_0})
 }
 where
 \EQ{ 
&V_{\ell_0}(r) :=\frac{2(\cos(2Q_{\ell_0})-1)}{r^2}\\
&N_{\ell_0}(r,u_{\ell_0}) := \cos(2Q_{\ell_0})\frac{(2ru_{\ell_0} - \sin(2ru_{\ell_0}))}{r^3}+2\sin(2 Q_{\ell_0}) \frac{\sin^2( ru_{\ell_0}) }{r^3}
}
Crucially, we remark that $\vec u_{\ell_0}(t)$ inherits the compactness property from $\vec \psi(t)$. Indeed, the trajectory
\ant{
\ti K:= \{\vec u_{\ell_0}(t)\mid t \in \R\}
}
 is pre-compact in $\dot{H}^1 \times L^2(\R^5_*)$. However, since we have assumed that $\ell_0 \neq 0$ we see that 
\EQ{ \label{uell not 0}
u_{\ell_0}(t,1) = \psi_0(t, 1) - Q_{\ell_0}(1) = - Q_{\ell_0}(1) \neq 0.
}
On the other hand, below we will show that $\vec u_{\ell_0} = ( u_{\ell_0}, \p_t u_{\ell_0}) = (0, 0)$ which contradicts~\eqref{uell not 0}.

\begin{lem}\label{ell not 0 lem} Suppose $\ell_0 \neq 0$. Let $\vec u(t)$ be as in Proposition~\ref{rigidity} and define $\vec u_{\ell_0}$ as in~\eqref{uell def}. Then $\vec u_{\ell_0} = (0, 0)$. 
\end{lem}
The argument that we will use to prove Lemma~\ref{ell not 0 lem} is nearly identical  to the one presented in the previous steps to reach the desired conclusion for $\ell_0 = 0$ and we omit many details here. 

We start by showing that $(\p_r u_{\ell_0, 0}, u_{\ell_0, 1})$ must be compactly supported. As before we can argue as in the proof of Lemma~\ref{key ineq lem}, by modifying \eqref{eq: u ell} inside the interior cone $\{(t, r) \mid 1 \le r \le R+ \abs{t}\}$, and using the linear exterior estimates in Proposition~\ref{linear prop} to produce the same type of inequality as~\eqref{key ineq}. 
\begin{lem} \label{key ineq lem ell}
There exists $R_0>1$ so that for all $R \ge R_0$ we have 
\EQ{\label{key ineq ell}
\| \pi^{\perp}_{R}\, \vec u_{\ell_0}\|_{\HH(r \ge R)}^2 &\lesssim R^{-22/3}\| \pi_{R}\, \vec u_{\ell_0}\|_{\HH(r \ge R)}^2 \\
&\quad + R^{-11/3}\| \pi_{R}\, \vec u_{\ell_0}\|_{\HH(r \ge R)}^4 +\| \pi_{R}\, \vec u_{\ell_0}\|_{\HH(r \ge R)}^6 
} 
where again $P(R):=\{ (k_1 r^{-3}, k_2 r^{-3})\mid k_1, k_2\in\R, \, r>R\}$,  $\pi_{R}\,$ denotes the orthogonal projection onto $P(R)$ and $\pi_{R}\,^\perp$ denotes the orthogonal projection onto the orthogonal complement of the plane $P(R)$
in $\HH(r>R;\R^5_*)$.   
\end{lem}
We remark that the proof of Lemma~\ref{key ineq lem ell} follows exactly as the proof of Lemma~\ref{key ineq lem} where we simply replace $Q$ with $ Q_{\ell_0}$  and $\vec u$ with $\vec u_{\ell_0}$ in the arguments given for the proof of Lemma~\ref{key ineq lem}.  We note that since the trajectory $\ti K$ is pre-compact in $\dot{H}^1 \times L^2(\R^5_*)$, $\vec u_{\ell_0}$ satisfies the conclusions of Corollary~\ref{compact}, namely for each $R> 1$ we have 
\ant{
\| \vec u_{\ell_0}(t)\|_{\HH(r \ge R + \abs{t})} \to 0 \mas \abs{t} \to \infty
}
where the condition $R>1$ allows the interchange of the norms $\HH= \dot{H}^1_0 \times L^2(\R^5_*)$ and  $\dot{H}^1 \times L^2(\R^5_*)$. With $(v_{\ell_0, 0}, v_{\ell_0, 1})$ defined as in~\eqref{v ell def} we can then conclude that for all $R>R_0$ large enough we have 
 \ali{\label{v ell ineq}
\int_R^{\infty} \left(\frac{1}{r} \p_r v_{\ell_0, 0}(r)\right)^2 \, dr + \int_R^{\I} (\p_r v_{\ell_0,1}( r))^2 \, dr   
&\lesssim R^{-\frac{31}{3}}v_{\ell_0,0}^2(R) + R^{-\frac{29}{3}} v_{\ell_0,0}^4(R)\\& + R^{-9}v_{\ell_0,0}^6(R)
+R^{-\frac{25}{3}}v_{\ell_0,1}^2(R)\\& + R^{-\frac{17}{3}} v_{\ell_0,1}^4(R) + R^{-3}v_{\ell_0,1}^6( R)\\
&\lesssim  R^{-7}(v_{\ell_0, 0}^2(R) + v_{\ell_0, 1}^2(R))
}
where the first inequality follows by rewriting~\eqref{key ineq ell} in terms of $ \vec v_{\ell_0} = (v_{\ell_0, 0}, v_{\ell_0, 1})$ and the last line following from the known decay estimates in \eqref{v ell def}. 
Next, mimicking the proof of Corollary~\ref{diff cor1} we  can again establish difference estimates using~\eqref{v ell ineq}. Indeed, for all $R_0 \le r \le r' \le 2r$ we have 
\ali{
&\abs{v_{\ell_0, 0}(r) - v_{\ell_0, 0}(r')}^2 \lesssim r^{-4}(v_{\ell_0, 0}^2(r) + v_{\ell_0, 1}^2(r))\\
&\abs{v_{\ell_0, 1}(r) - v_{\ell_0, 1}(r')}^2 \lesssim r^{-6}(v_{\ell_0, 0}^2(r) + v_{\ell_0, 1}^2(r))
}
In terms of the vector $\vec v_{\ell_0}  = (v_{\ell_0, 0},  v_{\ell_0, 1})$ we then have 
\ali{
&\abs{\vec v_{\ell_0}(r) - \vec v_{\ell_0}(r')} \lesssim r^{-2}\abs{\vec v_{\ell_0}(r)}
}
Hence for fixed $r_0 \ge R_0$ large enough we can deduce that 
\ant{
\abs{ \vec v_{\ell_0}(2^{n+1} r_0)} \ge \frac{3}{4} \abs{ \vec v_{\ell_0}(2^n r_0)}
}
Therefore for each $n$,
\ant{
\abs{ \vec v_{\ell_0}(2^{n} r_0)} \ge \left(\frac{3}{4}\right)^n \abs{ \vec v_{\ell_0}( r_0)}
}
On the other hand, by \eqref{v ell def} we have 
\ant{
\abs{ \vec v_{\ell_0}(2^{n} r_0)} \lesssim (2^nr_0)^{-1}
}
Combining the last two lines we see that
\ant{
\left(\frac{3}{2} \right)^n \abs{\vec v_{\ell_0}(r_0)} \lesssim 1,
}
which implies that $\vec v_{\ell_0}(r_0)=(0, 0)$. By~\eqref{v ell ineq} we can deduce that 
\ant{
\int_{r_0}^{\infty} \left(\frac{1}{r} \p_r v_{\ell_0, 0}(r)\right)^2 \, dr + \int_{r_0}^{\I} (\p_r v_{\ell_0,1}( r))^2 \, dr = 0
}
Therefore, 
\begin{multline*}
\| \vec u_{\ell_0}\|_{\HH(r \ge r_0)}^2= \\= \int_{r_0}^{\infty} \left(\frac{1}{r} \p_r v_{\ell_0, 0}(r)\right)^2 \, dr + \int_{r_0}^{\I} (\p_r v_{\ell_0,1}( r))^2 \, dr + 3r_0^{-3}v_{\ell_0, 0}^2(r_0) + r_0^{-1} v_{\ell_0, 1}^2(r_0) = 0
\end{multline*}
which means that $(\p_r u_{\ell_0, 0}, u_{\ell_0, 1})$ is compactly supported. We conclude by showing that $\vec u_{\ell_0} = (0, 0)$.

\begin{proof}[Proof of Lemma~\ref{ell not 0 lem}]The proof is nearly identical to the proof of Lemma~\ref{ell=0 lem}. Suppose $$(\p_ru_{\ell_0,0}, u_{\ell_0,1}) \neq (0, 0)$$ and we argue by contradiction.  
By the preceding arguments   $(\p_ru_{\ell_0,0}, u_{\ell_0,1})$ is compactly supported. Then we can define
 $\rho_0>1$ by 
 \ant{
 \rho_0 := \inf \{ \rho \, : \,  \|\vec u_{\ell_0}\|_{\HH(r \ge \rho)} =  0\}
 }
Let $ \e>0$ small to be determined below and find $1<\rho_1 < \rho_0$, $\rho_1= \rho_1( \e)$ so  that 
\ant{
0< \| \vec u_{\ell_0} \|_{\HH(r \ge \rho_1)} \le \e 
} 
We then have
\begin{multline} \label{pi pi perp ell}
\int_{\rho_1}^{\infty} \left(\frac{1}{r} \p_r v_{\ell_0,0}(r)\right)^2 \, dr + \int_{\rho_1}^{\I} (\p_r v_{\ell_0,1}( r))^2 \, dr+ 3\rho_1^{-3}v_{\ell_0,0}^2(\rho_1) + \rho_1^{-1}v_{\ell_0,1}^2( \rho_1)  =\\
= \| \pi_{\rho_1}^{\perp}\vec u_{\ell_0}\|_{\HH(r \ge \rho_1)}^2 +\| \pi_{\rho_1} \vec u_{\ell_0}\|_{\HH(r \ge \rho_1)}^2 = \|  \vec u_{\ell_0}\|_{\HH(r \ge \rho_1)}^2 < \e
\end{multline}
By \eqref{v ell ineq} we also have 
\begin{multline} \label{key ineq3 ell}
\int_{\rho_1}^{\infty} \left(\frac{1}{r} \p_r v_{\ell_0,0}(r)\right)^2 \, dr + \int_{\rho_1}^{\I} (\p_r v_{\ell_0,1}( r))^2 \, dr \lesssim \rho_1^{-\frac{31}{3}}v_{\ell_0, 0}^2(\rho_1) + \rho_1^{-\frac{29}{3}} v_{\ell_0, 0}^4(\rho_1)+ \\
 + \rho_1^{-9}v_{\ell_0,0}^6(\rho_1)
+\rho_1^{-\frac{25}{3}}v_{\ell_0, 1}^2(\rho_1) + \rho_1^{-\frac{17}{3}} v_{\ell_0, 1}^4(\rho_1) + \rho_1^{-3}v_{\ell_0, 1}^6(\rho_1)
\end{multline}
Arguing as in Corollary~\ref{diff cor2} and using the fact that $v_0( \rho_0) = v_1( \rho_0) = 0$ gives
\EQ{
\abs{v_{\ell_0, 0}( \rho_1)} = \abs{v_{\ell_0, 0}(\rho_1) - v_{\ell_0, 0}(\rho_0)} &\lesssim  \e \abs{v_{\ell_0, 0}(\rho_1)} +  \rho_1 \e\abs{v_{\ell_0, 1}(\rho_1)} \label{v0 diff2 ell}
}
and 
\EQ{
\abs{v_{\ell_0, 1}( \rho_1)}=\abs{v_{\ell_0, 1}(\rho_1) - v_{\ell_0, 1}(\rho_0)} &\lesssim \rho_1^{-1} \e \abs{v_{\ell_0, 0}(\rho_1)} +  \e\abs{v_{\ell_0, 1}(\rho_1)} \label{v1 diff2 ell}
}
Plugging \eqref{v0 diff2 ell} into \eqref{v1 diff2 ell} gives 
\ant{
\abs{v_{\ell_0, 1}( \rho_1)} \lesssim \rho_1^{-1} \e^2 \abs{v_{\ell_0, 0}(\rho_1)} + \e(1+\e)\abs{ v_{\ell_0, 1}( \rho_1)}
} 
which means that for $\e$ small enough we have 
\ali{\label{v1 v0 ell}
\abs{v_{\ell_0, 1}( \rho_1)} \lesssim \rho_1^{-1} \e^2 \abs{v_{\ell_0, 0}(\rho_1)}
}
Putting this estimate back into \eqref{v0 diff2 ell} we obtain
\ant{
\abs{v_{\ell_0, 0}( \rho_1)} \lesssim \e \abs{v_{\ell_0, 0}(\rho_1)} +  \e^3 \abs{v_{\ell_0, 0}(\rho_1)} \lesssim \e(1+ \e^2)\abs{ v_{\ell_0, 0} (\rho_1)}
}
which implies that $v_{\ell_0, 0}( \rho_1) = 0$ as long as $\e$ is chosen small enough. By \eqref{v1 v0 ell} we can conclude that $v_{\ell_0, 1}( \rho_1)=0$ as well. By \eqref{key ineq3 ell} and \eqref{pi pi perp ell} we then have that 
\ant{
\|\vec u_{\ell_0}\|_{\HH(r \ge \rho_1)} = 0
} 
which is a contradiction since $\rho_1 < \rho_0$.  Therefore, $(\p_ru_{\ell_0,0}, u_{\ell_0,1}) =(0, 0)$ Since $u_{\ell_0}(r) \to 0$ as $r \to \infty$ we can also conclude that $(u_{\ell_0,0}, u_{\ell_0,1}) = (0, 0)$.
\end{proof}

\subsection{Proof of Proposition~\ref{rigidity} and Proof of Theorem~\ref{main}}
 For clarity, we summarize what we have done in the proof of Proposition~\ref{rigidity}. 
\begin{proof}[Proof of Proposition~\ref{rigidity}]
Let $\vec u(t)$ be a solution to~\eqref{eq:u nl} and suppose that the trajectory 
\ant{
K=\{ \vec u(t) \mid t \in \R\}
}
is pre-compact in $ \HH$. We recall that 
\ant{
r\vec  u(t, r) = \vec \psi(t, r) - (Q_n(r), 0)
}
where $\vec \psi(t) \in \HH_n$ is a degree $n$ wave map, i.e., a solution to \eqref{eq:psi wm}. 
By Lemma~\ref{asymp lem} there exists $\ell_0 \in \R$ so that 
\begin{align} 
&\abs{r^3 u_0(r) - \ell_0 } = O(r^{-3}) \mas r \to \infty\\
&\abs{r \int_r^{\infty} u_1( \rho) \rho \, d\rho} = O(r^{-1}) \mas r \to \infty 
\end{align}
If $\ell_0 \neq 0$ then by Lemma~\ref{ell not 0 lem}, $\psi(0, r) = Q_{\ell_0}$ where $Q_{\ell_0}$ is defined in \eqref{Q ell def}. However, this is impossible since $Q_{\ell_0}(1) \neq 0$, which contradicts the Dirichlet boundary condition $\psi( t, 1) = 0$ for all $t \in \R$. 

Hence, $\ell_0 = 0$.  Then by Lemma~\ref{ell=0 lem} we can conclude that $\vec u(0)= (0, 0)$, which proves Proposition~\ref{rigidity}. 
\end{proof}

The proof of Theorem~\ref{main} is now complete. We conclude by summarizing the argument. 
\begin{proof}[Proof of Theorem~\ref{main}]
Suppose that Theorem~\ref{main} fails. Then by Proposition~\ref{prop:3} there exists a critical element, that is, a nonzero solution $\vec u_*(t) \in \HH$ to~\eqref{eq:u nl} such that the trajectory $K= \{ \vec u_*(t) \mid t \in \R\}$ is pre-compact in $\HH$. However, Proposition~\ref{rigidity} implies that any such solution is necessarily identically equal to $(0, 0)$, which contradicts the fact that the critical element $\vec u_*(t)$ is nonzero. 
\end{proof}

   
\bibliographystyle{plain}
\bibliography{researchbib}

 \bigskip

\centerline{\scshape Carlos Kenig, Andrew Lawrie, Wilhelm Schlag}
\medskip
{\footnotesize
 \centerline{Department of Mathematics, The University of Chicago}
\centerline{5734 South University Avenue, Chicago, IL 60615, U.S.A.}
\centerline{\email{cek@math.uchicago.edu, alawrie@math.uchicago.edu, schlag@math.uchicago.edu}}
} 

\end{document}